\documentclass[11pt,a4paper,reqno]{amsart}
\usepackage[english]{babel}
\usepackage[T1]{fontenc}
\usepackage{verbatim}
\usepackage{palatino}
\usepackage{amsmath}
\usepackage{mathabx}
\usepackage{amssymb}
\usepackage{amsthm}
\usepackage{amsfonts}
\usepackage{graphicx}
\usepackage{esint}
\usepackage{color}
\usepackage{mathtools}
\usepackage{overpic}

\usepackage[colorlinks = true, citecolor = black]{hyperref}
\author{Damian D{\k a}browski, Tuomas Orponen, and Hong Wang}
\title{How much can heavy lines cover?}
\address{Department of Mathematics and Statistics\\ University of Jyv\"askyl\"a,
P.O. Box 35 (MaD)\\
FI-40014 University of Jyv\"askyl\"a\\
Finland} \email{damian.m.dabrowski@jyu.fi} \email{tuomas.t.orponen@jyu.fi}
\address{Courant institute of mathematical sciences, New York University}
\email{Hong.Wang1991@gmail.com}
\date{\today}
\thanks{D.D is supported by the Research Council of Finland postdoctoral grant \emph{Quantitative rectifiability and harmonic measure beyond the Ahlfors-David-regular setting}, grant
		No. 347123. T.O. is supported by the Research Council of Finland via the project \emph{Approximate incidence geometry}, grant no. 355453, and by the European Research Council (ERC) under the European Union’s Horizon Europe research and innovation programme (grant agreement No 101087499). H.W. is supported by NSF CAREER DMS-2238818 and NSF DMS-2055544.}

\newcommand{\R}{\mathbb{R}}

\newcommand{\N}{\mathbb{N}}

\newcommand{\Z}{\mathbb{Z}}

\newcommand{\calT}{\mathcal{T}}

\newcommand{\calD}{\mathcal{D}}
\newcommand{\calH}{\mathcal{H}}

\newcommand{\calR}{\mathcal{R}}

\newcommand{\calK}{\mathcal{K}}

\newcommand{\Hd}{\dim_{\mathrm{H}}}

\newcommand{\calP}{\mathcal{P}}

\newcommand{\diam}{\operatorname{diam}}
\newcommand{\card}{\operatorname{card}}
\newcommand{\dist}{\operatorname{dist}}

\def\Barint_#1{\mathchoice
          {\mathop{\vrule width 6pt height 3 pt depth -2.5pt
                  \kern -8pt \intop}\nolimits_{#1}}%
          {\mathop{\vrule width 5pt height 3 pt depth -2.6pt
                  \kern -6pt \intop}\nolimits_{#1}}%
          {\mathop{\vrule width 5pt height 3 pt depth -2.6pt
                  \kern -6pt \intop}\nolimits_{#1}}%
          {\mathop{\vrule width 5pt height 3 pt depth -2.6pt
                  \kern -6pt \intop}\nolimits_{#1}}}

\numberwithin{equation}{section}

\theoremstyle{plain}
\newtheorem{thm}[equation]{Theorem}
\newtheorem*{"thm"}{"Theorem"}

\newtheorem{lemma}[equation]{Lemma}

\newtheorem{proposition}[equation]{Proposition}

\newtheorem{claim}[equation]{Claim}

\theoremstyle{definition}

\newtheorem{definition}[equation]{Definition}

\newtheorem{notation}[equation]{Notation}

\theoremstyle{remark}
\newtheorem{remark}[equation]{Remark}

\newcommand{\nref}[1]{(\hyperref[#1]{#1})}

\DeclareMathSymbol{\intop}  {\mathop}{mathx}{"B3}

\begin{document}

\begin{abstract} One formulation of Marstrand's slicing theorem is the following. Assume that $t \in (1,2]$, and $B \subset \R^{2}$ is a Borel set with $\mathcal{H}^{t}(B) < \infty$. Then, for almost all directions $e \in S^{1}$, $\mathcal{H}^{t}$ almost all of $B$ is covered by lines $\ell$ parallel to $e$ with $\Hd (B \cap \ell) = t - 1$.

We investigate the prospects of sharpening Marstrand's result in the following sense: in a generic direction $e \in S^{1}$, is it true that a strictly less than $t$-dimensional part of $B$ is covered by the heavy lines $\ell \subset \R^{2}$, namely those with $\dim_{\mathrm{H}} (B \cap \ell) > t - 1$? A positive answer for $t$-regular sets $B \subset \R^{2}$ was previously obtained by the first author.

The answer for general Borel sets turns out to be negative for $t \in (1,\tfrac{3}{2}]$ and positive for $t \in (\tfrac{3}{2},2]$. More precisely, the heavy lines can cover up to a $\min\{t,3 - t\}$ dimensional part of $B$ in a generic direction. We also consider the part of $B$ covered by the $s$-heavy lines, namely those with $\Hd (B \cap \ell) \geq s$ for $s > t - 1$. We establish a sharp answer to the question: how much can the $s$-heavy lines cover in a generic direction?

Finally, we identify a new class of sets called sub-uniformly distributed sets, which generalise Ahlfors-regular sets. Roughly speaking, these sets share the spatial uniformity of Ahlfors-regular sets, but pose no restrictions on uniformity across different scales. We then extend and sharpen the first author's previous result on Ahlfors-regular sets to the class of sub-uniformly distributed sets.

 \end{abstract}

\maketitle

\setcounter{tocdepth}{1}
\tableofcontents

\section{Introduction}

We start by stating Marstrand's slicing theorem \cite{Mar} as formulated in Mattila's book \cite[Theorem 6.9]{MR3617376}. For $e \in S^{1}$ and $z \in \R^{2}$, we use the notation $\ell_{e,z} := z + \mathrm{span}(e)$.
\begin{thm}[Marstrand, '54]\label{t:mar} Let $t \in (1,2]$, and let $B \subset \R^{2}$ be a Borel set with $\mathcal{H}^{t}(B) < \infty$. Then there exists a $\mathcal{H}^{1}$-null set $E \subset S^{1}$ such that the following holds for all $e \in S^{1} \, \setminus \, E$:
	\begin{equation}\label{form42} \Hd (B \cap \ell_{e,z}) = t - 1 \end{equation}
	for $\mathcal{H}^{t}$ a.e. $z \in B$. \end{thm}
It was shown by the second author \cite{MR3145914} that in fact $\Hd E \leq 2 - t$. Very roughly speaking, Marstrand's theorem says that the lines $\ell \subset \R^{2}$ failing \eqref{form42} are "exceptional". There are two ways in which \eqref{form42} can fail: either $\ell$ is \emph{light} or \emph{heavy}:
\begin{displaymath} \Hd (B \cap \ell) < t - 1 \quad \text{or} \quad \Hd (B \cap \ell) > t - 1. \end{displaymath}
In \cite[Section 6.4]{MR3617376}, Mattila proposes to study, how large a proportion of $B$ can be covered by such exceptional lines. Marstrand's theorem states, in a generic direction, that this proportion has vanishing $\mathcal{H}^{t}$ measure, but can one do better? The problem only makes sense in a generic direction: for example if $B = A \times A$ with $\Hd A = \tfrac{1}{2}\Hd B$, then all of $B$ is covered by heavy horizontal (or vertical) lines.

In the current paper, we focus on the problem of heavy lines. Let us briefly formalise our key notions. Given $A \subset \R^{2}$, we say that a line $\ell \subset \R^{2}$ is \emph{heavy} (for $A$) if
\begin{displaymath} \Hd (A \cap \ell) > \max\{\Hd A - 1,0\}. \end{displaymath}
More specifically, for a parameter $s > \max\{\Hd A - 1,0\}$, we say that a line $\ell \subset \R^{2}$ is \emph{$s$-heavy} if $\Hd (A \cap \ell) \geq s$.

\begin{definition}\label{p1} Let $A \subset \R^{2}$. For $e \in S^{1}$, let $\mathcal{H}(A,e,s)$ be the family of $s$-heavy lines parallel to $e$. The \emph{$s$-heavy part of $A$ in direction $e$} is the set
	\begin{displaymath} H(A,e,s) := \{z \in A : \ell_{e,z} \in \mathcal{H}(A,e,s)\}. \end{displaymath}
	We also define $\mathcal{H}(A,e)$ as the union of the families $\mathcal{H}(A,e,s)$ for $s > \max\{\Hd A - 1,0\}$, and finally
	\begin{displaymath}H(A,e) := \{z \in A : \ell_{e,z} \in \mathcal{H}(A,e)\}. \end{displaymath}
	Then, we set $\mathfrak{h}(A,s) := \mathrm{ess \, sup}_{e \in S^{1}} \Hd H(A,e,s)$ and $\mathfrak{h}(A) := \mathrm{ess \, sup}_{e \in S^{1}} \Hd H(A,e)$.  \end{definition}
The quantities $\mathfrak{h}(A,s)$ and $\mathfrak{h}(A)$ encode the answer to the question: how much of $A$ can (at most) be covered by the ($s$-)heavy lines in a generic direction? In a previous paper \cite[Theorem 1.3]{dabrowski2023visible}, the first author proved the following in the case where $A \subset \R^{2}$ is compact and Ahlfors-regular: 
\begin{equation}\label{form30} \mathfrak{h}(A) \leq 1. \end{equation} 
In particular, the value of (the upper bound for) $\mathfrak{h}(A)$ is independent of $\Hd A$, and becomes non-trivial if $\Hd A > 1$. Note that \eqref{form30} implies $\mathfrak{h}(A,e,s) \leq 1$ for all $s > \min\{\Hd A - 1,0\}$. When starting the research, it seemed reasonable to believe that
\begin{itemize}
	\item[(a) \phantomsection \label{a}] Ahlfors-regularity should not be necessary for \eqref{form30}, and 
	\item[(b) \phantomsection \label{b}] sharper estimates might hold for $s > \min\{\Hd A - 1,0\}$.
\end{itemize}

For readers familiar with the \emph{Furstenberg set problem}, it will not come as a surprise that Problem \ref{p1} is somehow related to Furstenberg sets. We clarify this connection presently.
\begin{definition} Let $s \in [0,1]$ and $t \in [0,2]$. A set $F \subset \R^{2}$ is an \emph{$(s,t)$-Furstenberg set} if there exists a family $\mathcal{L}$ of lines in $\R^{2}$ with $\Hd \mathcal{L} \geq t$ such that $\Hd (F \cap \ell) \geq s$ for all $\ell \in \mathcal{L}$. \end{definition}

Here $\Hd \mathcal{L}$ refers to the Hausdorff dimension of $\mathcal{L}$ viewed as a subset of $\mathcal{A}(2,1)$, the (metric) space of all affine lines in $\R^{2}$. A concrete metric in $\mathcal{A}(2,1)$ is given by the formula
\begin{displaymath} d_{\mathcal{A}(2,1)}(\ell_{1},\ell_{2}) := \|\pi_{L_{1}} - \pi_{L_{2}}\|_{\mathrm{op}} + |a_{1} - a_{2}|, \end{displaymath} 
where $\ell_{1} = L_{1} + a_{1}$ and $\ell_{2} = L_{2} + a_{2}$, and $L_{1},L_{2} \in \mathcal{G}(2,1)$ are $1$-dimensional subspaces of $\R^{2}$ parallel to $\ell_{1},\ell_{2}$, respectively, and $a_1\in L_1^\perp,$ $a_2\in L_2^\perp$.

In recent work, the third author with K. Ren \cite{2023arXiv230808819R} proved the following (sharp) lower bound for the Hausdorff dimension of $(s,t)$-Furstenberg sets:
\begin{thm}[Ren-Wang]\label{t:renWang} Let $F \subset \R^{2}$ be an $(s,t)$-Furstenberg set with $s \in (0,1]$ and $t \in [0,2]$. Then,
	\begin{displaymath} \Hd F \geq \min\{s + t,\tfrac{3s + t}{2},s + 1\}. \end{displaymath}
\end{thm}

In Theorem \ref{t:renWang} it is not necessary to assume any measurability of $F$. Using this information, we can make progress on the problem of heavy lines:
\begin{proposition}\label{propPartial} Let $A \subset \R^{2}$ be a set with $\Hd A = t \in [1,2]$. If $s > \tfrac{1}{3}(2t - 1)$, then $\mathcal{H}(A,s,e) = \emptyset$ for almost all $e \in S^{1}$, therefore $\mathfrak{h}(A,s) = 0$. For $t - 1 < s \leq \tfrac{1}{3}(2t - 1)$, we have
	\begin{displaymath} \mathfrak{h}(A,s) \leq \min\{2t - 3s,t\}. \end{displaymath}
	In particular $\mathfrak{h}(A) \leq \min\{3 - t,t\}$. \end{proposition}

\begin{proof} Assume first that $s > (2t - 1)/3$. Assume to the contrary that $\mathcal{H}(A,s,e) \neq \emptyset$ for positively many directions $e \in S^{1}$. This implies that there exist positively many $e \in S^{1}$ such that $\Hd (A \cap \ell_{e}) \geq s$ for at least one line $\ell_{e}$ parallel to $e$. Consequently $A$ is an $(s,1)$-Furstenberg set, and by Theorem \ref{t:renWang}
	\begin{displaymath} t = \Hd A \geq \tfrac{3s + 1}{2} > \tfrac{(2t - 1) + 1}{2} = t. \end{displaymath}
	This is a contradiction.
	
	Assume next $t - 1 < s < (2t - 1)/3$ (the endpoint $s = \tfrac{1}{3}(2t - 1)$ can be eventually treated by resorting to a sequence $s_{j} \nearrow \tfrac{1}{3}(2t - 1)$). 
	
	Let $\mathcal{H}(A,s) := \bigcup_{e \in S^{1}} \mathcal{H}(A,e,s)$ be the collection of all $s$-heavy lines in all directions. Abbreviate $\mathfrak{h} := \mathfrak{h}(A,s)$. Since the bound $\mathfrak{h} \leq t$ is trivial, it suffices to prove that $\mathfrak{h} \leq 2t - 3s$. We make a counter assumption: $\mathfrak{h} > 2t - 3s$. In particular $\mathfrak{h} > 1$ by the assumption $s < (2t - 1)/3$. By definitions, for any $\mathfrak{h}' \in (1,\mathfrak{h})$ there exist positively many directions $e \in S^{1}$ such that $\Hd H(A,e,s) > \mathfrak{h}'$. For all such directions $e \in S^{1}$, we have
	\begin{displaymath} \Hd \mathcal{H}(A,e,s) \geq \mathfrak{h}' - 1, \end{displaymath}
	because the lines in $\mathcal{H}(A,e,s)$ need to cover a set of dimension at least $\mathfrak{h}'$. Letting $\mathfrak{h}' \nearrow \mathfrak{h}$, it follows that $\Hd \mathcal{H}(A,s) \geq \mathfrak{h}$, and therefore $A$ is an $(s,\mathfrak{h})$-Furstenberg set. By Theorem \ref{t:renWang}, we deduce that $\Hd A \geq \min\{s + \mathfrak{h},\tfrac{3s + \mathfrak{h}}{2},s + 1\}$.
	
	All the three possibilities lead to a contradiction. If the minimum is the first term, then $\Hd A > s + (2t - 3s) = 2t - 2s \geq t$, because $s \leq (2t - 1)/3 \leq t/2$. If the minimum is the second term, then $\Hd A > \tfrac{1}{2}(3s + 2t - 3s) = t$ by the hypothesis $\mathfrak{h} > 2t - 3s$. Finally, if the minimum is the third term, then $\Hd A \geq s + 1 > t$, since $s > t - 1$.
	
	This completes the proof, except for final "in particular" part. However, the estimate for $\mathfrak{h}(A)$ follows from the cases $s > t - 1$ treated above by letting $s \searrow t - 1$.  \end{proof}

% According to \cite[Corollary 1.7]{fu2022incidence}, it holds $\Hd \mathcal{H}(A) \leq 3 - t$. This easily implies that $\Hd \mathcal{H}(A,e) \leq 2 - t$ for $\mathcal{H}^{1}$ almost every $e \in S^{1}$. Consequently 
%\begin{displaymath} \Hd H(A,e) \leq 1 + \Hd \mathcal{H}(A,e) \leq 3 - t \end{displaymath}
%for $\mathcal{H}^{1}$ almost every $e \in S^{1}$.
%\end{proof}
%Proposition \ref{fuRenProp} falls short from establishing \eqref{form30} for general Borel sets $A \subset \R^{2}$, and is particularly unsatisfactory for $t < \tfrac{3}{2}$. Clearly $\Hd H(A,e) \leq \Hd A = t$ for every $e \in S^{1}$, and this trivial bound beats "$3 - t$" when $t < \tfrac{3}{2}$. 

Let us then compare Proposition \ref{propPartial} to \eqref{form30}, namely the previously established result for Ahlfors-regular sets. Have we made progress with hypotheses \nref{a}-\nref{b}?

Let $A \subset \R^{2}$ be a Borel set with $\Hd A = t \in [1,2]$. If $A$ is Ahlfors-regular, \eqref{form30} states that $\mathfrak{h}(A) \leq 1$. In contrast, Proposition \ref{propPartial} only yields $\mathfrak{h}(A) \leq \min\{3 - t,t\}$. For instance, if $t \leq \tfrac{3}{2}$, Proposition \ref{propPartial} only returns the trivial bound $\mathfrak{h}(A) \leq t$. So, we have made virtually no progress in confirming hypothesis \nref{a}, especially if $\Hd A \leq \tfrac{3}{2}$.

How about hypothesis \nref{b}? Indeed, if $s > \tfrac{1}{3}(2t - 1)$, we have shown that $\mathfrak{h}(A,s) = 0$, which is certainly sharper than \eqref{form30}. However, for $s = \tfrac{1}{3}(2t - 1)$ Proposition \ref{propPartial} only promises that $\mathfrak{h}(A,s) \leq 1$. So, we are unable to improve on \eqref{form30}, unless $s > \tfrac{1}{3}(2t - 1)$. Thus, for Borel sets, Proposition \ref{propPartial} verifies neither \nref{a} nor \nref{b}.

It turns out that the bound in Proposition \ref{propPartial} is sharp:

\begin{thm}\label{mainExampleIntro} For every $t \in (1,2]$ and $s \in [t - 1,\tfrac{1}{3}(2t - 1)]$ there exists a compact set $K \subset \R^{2}$ such that $\Hd K = t$ and $\mathfrak{h}(K,s) = \min\{2t - 3s,t\}$. In fact, $K$ can be selected so that
	\begin{displaymath} \Hd H(K,e,s) = \min\{2t - 3s,t\} \end{displaymath}
	for every $e \in S^{1}$. \end{thm}

\begin{remark} We can now answer the question on how much Theorem \ref{t:mar} can be sharpened for heavy lines. If $\Hd A = t \in (\tfrac{3}{2},2]$, then the heavy lines in $\mathcal{H}^{1}$ almost every direction can only cover a $(3 - t)$-dimensional set (by Proposition \ref{propPartial}), where $3 - t < t$.
	
	If $t \in (1,\tfrac{3}{2}]$, the heavy lines in every direction may cover a $t$-dimensional set. For $t \in (1,\tfrac{3}{2})$, \emph{a fortiori}, the $\tfrac{t}{3}$-heavy lines in every direction may cover a $t$-dimensional set.
	
	(For $t = \tfrac{3}{2}$, the example showing that heavy lines can cover a $t$-dimensional set is constructed as a union of the sets $K = K_{t,s_{j}}$ in Theorem \ref{mainExampleIntro} with $s_{j} = \tfrac{1}{2} + \tfrac{1}{j}$.) \end{remark}

\begin{remark} Proposition \ref{propPartial} and Theorem \ref{mainExampleIntro} tell us something about $\Hd H(A,e)$ for generic $e \in S^{1}$. A related question concerns $\Hd \mathcal{H}(A,e)$. For this problem, a sharp answer was given earlier by Fu and Ren \cite[Corollary 1.7]{fu2022incidence}. In fact $\Hd \mathcal{H}(A,e) \leq 2 - t$ for $\mathcal{H}^{1}$ almost every $e \in S^{1}$. An easy generalisation of their argument shows, more generally, that if $s \in (t - 1,1]$, then $\Hd \mathcal{H}(A,e,s) \leq 1 - s$ for $\mathcal{H}^{1}$ almost every $e \in S^{1}$. 

Given these bounds, one may be tempted to pursue the following formal "proof" of the inequality $\Hd H(A,e) \leq 1$. For $s \in (t - 1,1]$ and $\epsilon > 0$, we "sum up" the dimension of the line family and the upper bound for the heaviness to obtain
\begin{equation}\label{falseFubini} \Hd [H(A,e,s) \, \setminus \, H(A,e,s + \epsilon)] \leq (1 - s) + (s + \epsilon) = 1 + \epsilon. \end{equation}
Then, we vary $s \in (t - 1,1]$ and finally let $\epsilon \to 0$ to deduce that $\Hd H(A,e) \leq 1$. Of course \eqref{falseFubini} is suspicious, since the argument relies on a "Fubini theorem for Hausdorff dimension" which is generally false. In fact, Theorem \ref{mainExampleIntro} shows that this argument is impossible to make rigorous, at least for general compact sets.
\end{remark}

\subsection{Sub-uniformly distributed sets} Theorem \ref{mainExampleIntro} says that Proposition \ref{propPartial} is sharp in the class of Borel, or even compact, sets. For Ahlfors-regular sets \eqref{form30} says something much stronger. Which property of Ahlfors-regular sets explains this discrepancy? 

The examples constructed for Theorem \ref{mainExampleIntro} have the form $K = K_{1} \cup K_{2}$, where both $K_{1},K_{2}$ are Cantor-type sets of dimension $t$, but with wildly different "branching" behaviour. Informally speaking, the set $K_{1}$ looks $2$-dimensional between certain scales $[\delta_{n},\Delta_{n}]$ and $t$-dimensional between other scales $[\Delta_{n + 1},\delta_{n}]$. The set $K_{2}$ has the same properties, but with the roles of the scales reversed. In particular, it would be ill-defined to say that $K_{1} \cup K_{2}$ looks $s$-dimensional between the scales $[\delta_{n},\Delta_{n}]$, for any $s \in [0,2]$.

To improve on Theorem \ref{mainExampleIntro}, we introduce the following definition which aims to (i) extend Ahlfors-regular sets, and (ii) rule out the adverse behaviour described above:

\begin{definition}[Sub-uniformly distributed sets]\label{def:subUniform} We say that a bounded set $K \subset \R^{d}$ is \emph{sub-uniformly distributed} if there exists a constant $C > 0$ such that
	\begin{equation}\label{form43} |K|_{R} \cdot |K \cap Q|_{r} \le C|K|_{r}, \qquad Q \in \mathcal{D}_{R}(K), \, 0 < r \leq R < \infty. \end{equation}
\end{definition}

\begin{remark} At first we considered the slightly stronger definition of \emph{uniformly distributed sets} which would otherwise be defined as above, except that we require a $2$-sided estimate $|K \cap Q|_{r} \sim |K|_{r}/|K|_{R}$. The caveat of the stronger definition is that the uniformity of a set might depend on the choice of a dyadic system, or whether the set $Q$ is taken to be an $R$-disc or an $R$-square. The notion of sub-uniformly distributed sets is blind to such nuances, yet strong enough for our purposes. 
	
	Clearly $t$-Ahlfors-regular sets are sub-uniformly distributed: $|K \cap Q|_{r} \lesssim (R/r)^{t} \sim |K|_{r}/|K|_{R}$ for all $Q \in \mathcal{D}_{R}$ and $0 < r \leq R \leq \diam(K)$. For $R > \diam(K)$ the estimate \eqref{form43} follows simply from $|K|_{R} \sim 1$. However, sub-uniformly distributed sets are much more general than Ahlfors-regular sets: e.g. any Cantor type set obtained by replacing squares of level $n$ by $C(n) \in \N$ squares squares of level $n + 1$ is sub-uniformly distributed.
	
	The union of two sub-uniformly distributed sets is generally not sub-uniformly distributed. Indeed, the sets in Theorem \ref{mainExampleIntro} can be written as a union of two sub-uniformly distributed sets, and they fail the conclusion of Theorem \ref{thm:dim-estimate}.
	
	Finally, we mention that the notion of (sub-)uniformly distributed sets was inspired by the notion of $\{\Delta_{j}\}_{n = 0}^{n}$-uniform sets which has proved useful in dealing with Furstenberg sets and related problems in approximate incidence geometry \cite{OS23,2023arXiv230808819R,Sh}. \end{remark}

We then arrive at our main result for sub-uniformly distributed sets:

\begin{thm}\label{thm:dim-estimate}
	Let $K \subset \R^{2}$ be compact and sub-uniformly distributed with $\dim_{\mathrm{H}}(K)=t\in (0,2)$. For $\max\{t-1,0\}<s\leq 1$ we have
	\begin{equation}\label{eq:dim-estim}
	\mathfrak{h}(K,s) \leq \max\{t-s,0\}.
	\end{equation}
	In particular, $\mathfrak{h}(K) \leq 1$.
\end{thm}
In fact, we prove this estimate for a slightly larger ``heavy part'', where heaviness is measured using box-counting dimension; see Proposition \ref{prop:Ahlfors} for details.
\begin{remark}
	For $s>\frac{t}{2}$ we see from \eqref{eq:dim-estim} that $\dim_{\mathrm{H}}H(K,e,s)\le \frac{t}{2}<s$ for a.e. $e\in S^1$. Since a single $s$-heavy line satisfies $\dim_{\mathrm{H}}(K\cap \ell)\ge s$, it follows that for a.e. $e\in S^1$ there are no $s$-heavy lines parallel to $e$, i.e. $\mathcal{H}(K,e,s)=\varnothing$. Compare this to Proposition \ref{propPartial}, which stated that typically there are no $s$-heavy lines as soon as $s>\frac{1}{3}(2t-1)$. 
	
	Since $\frac{t}{2}\ge \frac{1}{3}(2t-1)$, the threshold of Proposition \ref{propPartial} is lower than the one we can obtain from Theorem \ref{thm:dim-estimate}. In other words, for $s> \frac{1}{3}(2t-1)$ the estimate from Proposition \ref{propPartial} (valid for general Borel sets) beats Theorem \ref{thm:dim-estimate}. On the other hand, for $t-1<s\le \frac{1}{3}(2t-1)$ the bound from Theorem \ref{thm:dim-estimate} beats Proposition \ref{propPartial}. Regardless, in this range we suspect that the bound \eqref{eq:dim-estim} is not sharp in the class of sub-uniformly distributed sets, let alone Ahlfors-regular sets. \end{remark}

\subsection{Further literature} If we restrict to sets with an underlying dynamical system, there are significantly stronger answers to the heavy lines problem than we presented above for Ahlfors-regular and sub-uniformly distributed sets. For instance, solving an old conjecture of Furstenberg \cite{MR0354562}, Shmerkin \cite{Sh} and Wu \cite{MR3961082} independently established the following. Assume that $K = A \times B$, where $A$ is $\times p$-invariant, $B$ is $\times q$-invariant, and $\log p/\log q \notin \mathbb{Q}$. Then $H(K,e) = \emptyset$ for all $e \in S^{1} \, \setminus \, \mathrm{span}\{(0,1),(1,0)\}$. A similar conclusion holds (without any exceptional directions) if $K \subset \R^{2}$ is a self-similar set where one of the generators contains an irrational rotation, see \cite[Theorem 1.6]{MR3961082}. For more recent work related to Furstenberg's intersection conjecture, see for example \cite{MR4093958,MR4565614,2023arXiv230317197A,MR4357267,MR4302170}

Finally, we refer the reader to the recent survey of Mattila \cite{mca28020049} on various slicing problems.

\section{Notation and preliminaries}\label{s:prelim}
\begin{notation}
	We write $f\lesssim g$ if there exists an absolute constant $C>0$ such that $f\le C g$. If $C$ depends on some parameter $\epsilon$, we will write $f\lesssim_{\epsilon} g$. In case $f\lesssim g\lesssim f$ we write $f\sim g$, while $f\sim_{\epsilon} g$ denotes $f\lesssim_{\epsilon} g\lesssim_{\epsilon} f$.
\end{notation}

\begin{notation}[Families of dyadic cubes] The notation $\mathcal{D}_{\delta}(\R^{d})$ will refer to all the standard dyadic cubes of $\R^{d}$. More generally, if $P \subset \R^{d}$ or $P \subset \mathcal{D}_{\delta}(\R^{d})$, and $\Delta \in 2^{-\N}$, we will use the notation 
\begin{displaymath} P_{\Delta} := \mathcal{D}_{\Delta}(P) := \{\mathbf{p} \in \mathcal{D}_{\Delta}(\R^{d}) : \mathbf{p} \cap P \neq \emptyset\}. \end{displaymath}
We will also write $|P|_\Delta\coloneqq|P_\Delta|$ to denote the (dyadic) $\Delta$-covering number of $P$.

In the special case $P = [0,1)^{2}$ we abbreviate
\begin{displaymath} \mathcal{D}_{\delta} := \mathcal{D}_{\delta}([0,1)^{2}). \end{displaymath}
If $\mathbf{p} \in \mathcal{D}_{\Delta}(\R^{d})$ and $\mathcal{P} \subset \mathcal{D}_{\delta}(\R^{d})$ with $\delta \leq \Delta \leq 1$, we will also (ab-)use the notation $\mathcal{P} \cap \mathbf{p} := \{p \in \mathcal{P} : p \subset \mathbf{p}\}$. \end{notation}

\begin{notation}\label{not1} Let $\delta \in 2^{-\N}$. The \emph{dyadic $\delta$-tubes} $\mathcal{T}^{\delta}$ are the images of elements of $\mathcal{D}_{\delta}$ under the "point-line duality map" 
\begin{displaymath} \mathbf{D}(a,b) := \{(x,y) \in \R^{2} : y = ax + b\}. \end{displaymath}
More precisely, each $T \in \mathcal{T}^{\delta}$ has the form $T = \cup\mathbf{D}(p)$ for some $p = p_{T} \in \mathcal{D}_{\delta}$. (It is a choice of normalisation, not a typo, that we only consider images of $\mathcal{D}_{\delta}$ instead of $\mathcal{D}_{\delta}(\R^{d})$.) For more information, see \cite[Section 2.3]{OS23}. If $\mathbf{T} \in \mathcal{T}^{\Delta}$, and $\mathcal{T} \subset \mathcal{T}^{\delta}$, we will (ab-)use the notation $\mathcal{T} \cap \mathbf{T} := \{T \in \mathcal{T} : T \subset \mathbf{T}\}$. 
\end{notation}

\begin{notation} For every dyadic cube $p \in \mathcal{D}_{\delta}$ we associate the disc $B_{p}$ which is concentric with $p$ and has diameter $10\delta$ (in particular $p \subset B_{p}$). Similarly, for every dyadic tube $T = \cup \mathbf{D}([a + \delta) \times [b + \delta)) \in \mathcal{T}^{\delta}$, $\delta \in 2^{-\N}$, we associate the "standard" tube $\mathbb{T}_{T}$ which has width $10\delta$ and whose core line is $\mathbf{D}(a,b)$. In particular, 
\begin{displaymath} T \cap B(1) \subset \mathbb{T}_{T}. \end{displaymath}
The definitions of $B_{p}$ and $\mathbb{T}_{T}$ have been posed so that the inclusion ordering of the discs and ordinary tubes respects the inclusion ordering of the dyadic objects: if $p,\mathbf{p} \in \bigcup_{\delta}\mathcal{D}_{\delta}$ with $p \subset \mathbf{p}$, then $B_{p} \subset B_{\mathbf{p}}$, and similarly if $T,\mathbf{T} \in \bigcup_{\delta}\mathcal{T}^{\delta}$ with $T \subset \mathbf{T}$, then $\mathbb{T}_{T}\cap B(2) \subset \mathbb{T}_{\mathbf{T}}\cap B(2)$. %\textcolor{red}{I did the computation, and one actually only has $\mathbb{T}_{T}\cap B(2) \subset \mathbb{T}_{\mathbf{T}}\cap B(2)$ and not $\mathbb{T}_{T} \subset \mathbb{T}_{\mathbf{T}}$, but this is good enough for our purposes. Is it worth writing down the computation?}

In the sequel, dyadic tubes are denoted with font $T,T'$ and standard tubes with font $\mathbb{T},\mathbb{T}'$.
\end{notation}

\begin{notation} Whenever $\mathcal{P} \subset \mathcal{D}_{\delta}$ and $\mathcal{T} \subset \mathcal{T}^{\delta}$, we will consider the set of \emph{incidences}
\begin{equation}\label{form4} \mathcal{I}(\mathcal{P},\mathcal{T}) := \{(p,T) \in \mathcal{P} \times \mathcal{T} : B_{p} \cap \mathbb{T}_{T} \neq \emptyset\}. \end{equation} 
Thus, the notion of incidences between dyadic squares and tubes is defined as a set of incidences between a family of balls and ordinary tubes. This will facilitate using existing estimates on incidences which generally concern balls and ordinary tubes. The following property is worth recording:
\begin{equation}\label{form11} p \subset \mathbf{p}, \, T \subset \mathbf{T} \text{ and } B_{p} \cap \mathbb{T}_{T} \cap B(2) \neq \emptyset \quad \Longrightarrow \quad B_{\mathbf{p}} \cap \mathbb{T}_{\mathbf{T}} \neq \emptyset. \end{equation}
This follows immediately from $B_{p} \subset B_{\mathbf{p}}$ and $\mathbb{T}_{T}\cap B(2) \subset \mathbb{T}_{\mathbf{T}}\cap B(2)$.

It will occasionally be fruitful to view the set of incidences $\mathcal{I}(\mathcal{P},\mathcal{T})$  as a subset of $\mathcal{D}_{\delta}(\R^{4})$ by identifying the pair $(p,T)$ with the $\delta$-cube 
\begin{displaymath} p \otimes T := p \times p_{T} \in \mathcal{D}_{\delta}(\R^{4}). \end{displaymath} 
If $\mathcal{H} \subset \mathcal{I}(\mathcal{P},\mathcal{T})$ and $\Delta \in 2^{-\N} \cap [\delta,1]$, the notation $|\mathcal{H}|_{\Delta}$ refers to the number of elements in $\mathcal{D}_{\Delta}(\R^{4})$ intersecting $\mathcal{H}$.

 \end{notation}

\begin{lemma}\label{lemma2} Let $\delta \in 2^{-\N} \cap (0,\tfrac{1}{100}]$, $\mathcal{P} \subset \mathcal{D}_{\delta}$ and $\mathcal{T} \subset \mathcal{T}^{\delta}$. Then,
\begin{displaymath} (\mathcal{I}(\mathcal{P},\mathcal{T}))_{\Delta} \subset \mathcal{I}(\mathcal{P}_{\Delta},\mathcal{T}_{\Delta}), \qquad \Delta \in 2^{-\N} \cap [\delta,1]. \end{displaymath}
\end{lemma}

\begin{proof} Let $Q \in (\mathcal{I}(\mathcal{P},\mathcal{T}))_{\Delta} \subset \mathcal{D}_{\Delta}(\R^{4})$. Thus $Q$ contains at least one $\delta$-cube $p \otimes T$ with $p \in \mathcal{P}$, $T \in \mathcal{T}$, and $B_{p} \cap \mathbb{T}_{T} \neq \emptyset$. The hypothesis $\delta \leq \tfrac{1}{100}$ ensures that $B_{p} \subset B(2)$, and therefore $B_{p} \cap \mathbb{T}_{T} \cap B(2) \neq \emptyset$. Now the dyadic $\Delta$-parents of $p$ and $T$ also satisfy $B_{\mathbf{p}} \cap \mathbb{T}_{\mathbf{T}} \neq \emptyset$, as recorded in \eqref{form11}, so $\mathbf{p} \otimes \mathbf{T} \in \mathcal{I}(\mathcal{P}_{\Delta},\mathcal{T}_{\Delta})$. But since $\mathbf{p} \otimes \mathbf{T} \in \mathcal{D}_{\Delta}(\R^{4})$ is a dyadic $\Delta$-cube containing $p \otimes T$, we must have $Q = \mathbf{p} \otimes \mathbf{T} \in \mathcal{I}(\mathcal{P}_{\Delta},\mathcal{T}_{\Delta})$.
\end{proof}

\begin{remark} The converse inclusion $\mathcal{I}(\mathcal{P}_{\Delta},\mathcal{T}_{\Delta}) \subset (\mathcal{I}(\mathcal{P},\mathcal{T}))_{\Delta}$ is false in general. It is possible that $B_{\mathbf{p}} \cap \mathbb{T}_{\mathbf{T}} \neq \emptyset$ for some $\mathbf{p} \in \mathcal{P}_{\Delta}$ and $\mathbf{T} \in \mathcal{T}_{\Delta}$, and thus $\mathbf{p} \otimes \mathbf{T} \in \mathcal{I}(\mathcal{P}_{\Delta},\mathcal{T}_{\Delta})$, but there exists no pair $p \in \mathcal{P} \cap \mathbf{p}$ and $T \in \mathcal{T} \cap \mathbf{T}$ such that $B_{p} \cap \mathbb{T}_{T} \neq \emptyset$. In this case $\mathbf{p} \otimes \mathbf{T} \notin (\mathcal{I}(\mathcal{P},\mathcal{T}))_{\Delta}$. \end{remark}

\subsection{The high-low method} A main technical tool for us will be the \emph{high-low method}, pioneered by Guth, Solomon, and Wang in \cite{GSW}. We will employ the technique in the form formalised by Bradshaw \cite[Proposition 2.1]{2020arXiv200911765B}:
\begin{proposition}\label{prop:highLow} Fix $\epsilon,\delta \in (0,\tfrac{1}{2}]$. Let $\mathcal{B}$ be a family of $\delta$-balls contained in $B(1) \subset \R^{2}$, and let $\mathcal{T}$ be a family of (ordinary) $\delta$-tubes. Fix $A \in [\delta^{-\epsilon},\delta^{-1}]$. Then,
\begin{equation}\label{form2} |\mathcal{I}(\mathcal{B},\mathcal{T})| \lesssim_{\epsilon} \sqrt{A\delta^{-1}|\mathcal{B}||\mathcal{T}|} + \delta^{-\epsilon}A^{-1}|\mathcal{I}(\mathcal{B}^{A},\mathcal{T}^{A})|, \end{equation}
where $\mathcal{B}^{A} = \{B^{A} : B \in \mathcal{B}\}$ and $\mathcal{T}^{A} = \{\mathbb{T}^{A} : \mathbb{T} \in \mathcal{T}\}$ consist of the $A$-thickenings of the balls and tubes in $\mathcal{B}$ and $\mathcal{T}$, respectively.
\end{proposition}

This proposition is known as the "high-low method", because its proof shows that the first term in \eqref{form2} bounds the left hand side in a "high frequency dominated case", whereas the second term bounds the left hand side in a "low frequency dominated" case. 

\begin{remark} We make a few clarifying remarks about Proposition \ref{prop:highLow}. First, as can be expected, $\mathcal{I}(\mathcal{B},\mathcal{T}) = \{(B,\mathbb{T}) \in \mathcal{B} \times \mathcal{T} : B \cap \mathbb{T} \neq \emptyset\}$. The families $\mathcal{B}^{A}$ and $\mathcal{T}^{A}$ consist of $(A\delta)$-balls and ordinary $(A\delta)$-tubes, but -- in a typical application -- these families will be far from $(A\delta)$-separated. The right intuition is that the families $\mathcal{B}$ and $\mathcal{T}$ typically consist of $\delta$-separated objects, but when the objects in these families are thickened by $A$, they tend to have heavy overlap.
\end{remark}

\section{A new incidence estimate}
In order to prove Theorem \ref{thm:dim-estimate} we establish a new incidence estimate. It is valid for discretized variants of sub-uniformly distributed sets, as defined below.
\begin{definition}
	Let $\mathbf{C} > 0$ and let $\{\Delta_n,\dots,\Delta_0\}\subset 2^{-\mathbb{N}}$ be an increasing sequence of scales. Suppose that $\mathcal{P}$ is either a subset of $B(0,1)$ or a subfamily of $\mathcal{D}_{\delta}([0,1)^2)$ with $\delta\le \Delta_n$. We say that $\calP$ is \emph{sub-$\{\Delta_j\}_{j=0}^n$-uniform with constant $\mathbf{C}$} if for every $j \in \{0,\dots, n-1\}$
	\begin{equation*}
	|\mathcal{P}|_{\Delta_{j}} \cdot |\calP \cap \mathbf{p}|_{\Delta_{j+1}}\le \mathbf{C} |\calP|_{\Delta_{j+1}}, \qquad \mathbf{p} \in \mathcal{D}_{\Delta_{j}}(\mathcal{P}).
	\end{equation*}
	%If the constant $C$ is irrelevant, we will not mention it.
\end{definition}

%We say that $\mathcal{P}$ is \emph{$\{\Delta_j\}_{j=0}^n$-uniform with constant $C$} if for every $j=0,\dots, n-1$ and every $\mathbf{p}\in \calP_{\Delta_j}$ \textcolor{cyan}{Is the notion of uniform sets needed for something?}
	%\begin{equation*}
	%C^{-1} \frac{|\calP_{\Delta_{j+1}}|}{|\calP_{\Delta_{j}}|}\le |\mathbf{p}\cap \calP_{\Delta_{j+1}}|\le C \frac{|\calP_{\Delta_{j+1}}|}{|\calP_{\Delta_{j}}|}.
	%\end{equation*}

\begin{proposition}\label{prop1} For every $\mathbf{C} > 0$, $\kappa \in (0,1)$ and $\eta \in (0,\tfrac{\kappa}{2}]$ the following holds for all $\delta \in 2^{-\N}$ small enough such that also $S := \delta^{-\eta} \in 2^{-\N}$.  Consider the scale sequence $\{\delta,S\delta,S^{2}\delta,\ldots,1\} = \{\Delta_{n},\Delta_{n - 1},\ldots,\Delta_{0}\}$. Let $\mathcal{T} \subset \mathcal{T}^{\delta}$ and $\mathcal{P} \subset \mathcal{D}_{\delta}$ be sub-$\{\Delta_j\}_{j=0}^n$-uniform with constant $\mathbf{C}$. Assume that
\begin{equation*}
|\mathcal{I}(\mathcal{P},\mathcal{T})| \geq \delta^{1 - \kappa}|\mathcal{P}||\mathcal{T}|.
\end{equation*}
Then, there exists a scale $\Delta \in \{\Delta_{1},\ldots,\Delta_{n}\} \subset [\delta,\delta^{\eta}]$ such that
\begin{equation}\label{form3} |\mathcal{I}(\mathcal{P}_{\Delta},\mathcal{T}_{\Delta})| \lesssim_{\mathbf{C},\eta} \delta^{\kappa-2\eta} \Delta^{-2}. \end{equation}
Here $\mathcal{P}_{\Delta} := \mathcal{D}_{\Delta}(\mathcal{P})$ and $\mathcal{T}_{\Delta} = \mathcal{T}^{\Delta}(\mathcal{T})$ are the dyadic $\Delta$-covers of $\mathcal{P}$ and $\mathcal{T}$, respectively.
\end{proposition}

%\begin{remark} In fact, the proof works for $\eta,\kappa \in (0,1)$ independent of each other, but since it is possible that $\Delta = \delta^{\eta}$, the information in \eqref{form3} is only useful when $\eta \leq \kappa$. \end{remark}

\begin{proof} The main idea is to apply the high-low method along the scale sequence $\{\Delta_{j}\}_{j = 0}^{n}$ (in increasing order) and observe that that the \emph{incidence quotient}
\begin{displaymath} \iota(\Delta) := \frac{|\mathcal{I}(\mathcal{P}_{\Delta},\mathcal{T}_{\Delta})|}{\Delta|\mathcal{P}_{\Delta}||\mathcal{T}_{\Delta}|}. \end{displaymath}
is (roughly) non-decreasing as long as the low "low case" of Proposition \ref{prop:highLow} occurs. Eventually, this will show that the high case must occur -- indeed well before scale "$1$" -- and this is the scale at which \eqref{form3} holds. 

\subsubsection*{Monotonicity of incidence quotients} Fix $\Delta = \Delta_{j}$ with $j \in \{1,\ldots,n\}$, in particular $\Delta \in [\delta,\delta^{\eta}]$. Let $c > 0$ be an absolute constant to be determined a little later. We apply Proposition \ref{prop:highLow} at scale $\Delta$ with parameters $\epsilon=\eta^{2}$ and $A := cS$ and to the families $\{B_{p} : p \in \mathcal{P}_{\Delta}\}$ and $\{\mathbb{T}_{T} : T \in \mathcal{T}_{\Delta}\}$ (recall that $\mathcal{I}(\mathcal{P}_{\Delta},\mathcal{T}_{\Delta})$ is defined using these balls and ordinary tubes). The proposition requires $A \in [\Delta^{-\eta^{2}},\Delta^{-1}]$, and this is satisfied since $A \leq S = \delta^{-\eta} \leq \Delta^{-1}$.

The conclusion is that, for a suitable constant $C = C(\eta) > 0$, either
\begin{equation}\label{form5} |\mathcal{I}(\mathcal{P}_{\Delta},\mathcal{T}_{\Delta})| \leq C\sqrt{A\Delta^{-1}|\mathcal{P}_{\Delta}||\mathcal{T}_{\Delta}|} \end{equation}
or
\begin{equation}\label{form6} |\mathcal{I}(\mathcal{P}_{\Delta},\mathcal{T}_{\Delta})| \leq C\Delta^{-\eta^{2}}A^{-1}|\mathcal{I}(\mathcal{P}_{\Delta}^{A},\mathcal{T}_{\Delta}^{A})|.  \end{equation}
Here $\mathcal{P}_{\Delta}^{A}$ and $\mathcal{T}_{\Delta}^{A}$ refer to the families of ordinary balls (resp. tubes) of diameter (resp. width) $10A\Delta$ which are the $A$-times enlargements of the families $\{B_{p} : p \in \mathcal{P}_{\Delta}\}$ and $\{\mathbb{T}_{T} : T \in \mathcal{T}_{\Delta}\}$.

If \eqref{form5} holds, we say that \emph{the high case occurs at scale $\Delta$}. If \eqref{form5} fails (therefore \eqref{form6} holds), we say that \emph{the low case occurs at scale $\Delta$}. (We point out that these definitions depend on our choice of $\eta$.) We claim that if the low case occurs at scale $\Delta$, then 
\begin{equation}\label{form9} \iota(\Delta) \lesssim_{\mathbf{C},\eta} \Delta^{-\eta^{2}}\iota(S\Delta). \end{equation}

Recall that both $\mathcal{P}$ and $\mathcal{T}$ are assumed to be sub-$\{\Delta_{j}\}_{j = 0}^{n}$-uniform. Given $\Delta\in\{\Delta_n,\dots,\Delta_{1}\}$ let $M=|\mathcal{P}_{\Delta}|/|\mathcal{P}_{S\Delta}|$ and $N=|\mathcal{T}_{\Delta}|/|\mathcal{T}_{S\Delta}|$, so that
\begin{equation}\label{form8} 
	|\mathcal{P}_{\Delta} \cap \mathbf{p}| \leq \mathbf{C}M \quad \text{and} \quad |\mathcal{T}_{\Delta} \cap \mathbf{T}| \leq \mathbf{C}N 
\end{equation}
for all $\mathbf{p} \in \mathcal{D}_{S\Delta}(\mathcal{P})$ and $\mathbf{T} \in \mathcal{D}_{S\Delta}(\mathcal{T})$. We now claim that
\begin{equation}\label{form7} 
	|\mathcal{I}(\mathcal{P}^{A}_{\Delta},\mathcal{T}^{A}_{\Delta})| \leq \mathbf{C}^{2}MN \cdot |\mathcal{I}(\mathcal{P}_{S\Delta},\mathcal{T}_{S\Delta})|. 
\end{equation} 
To see this, let $(p,T) \in \mathcal{P}_{\Delta} \times \mathcal{T}_{\Delta}$ be such that $AB_{p} \cap A\mathbb{T}_{T} \neq \emptyset$ (these are the pairs we are counting on the left hand side of \eqref{form7}). Let $\mathbf{p} \in \mathcal{D}_{S\Delta}(\mathcal{P})$ and $\mathbf{T} \in \mathcal{T}_{S\Delta}(\mathcal{T})$ be the dyadic parents of $p$ and $T$, respectively. We claim that $B_{\mathbf{p}} \cap \mathbb{T}_{\mathbf{T}} \neq \emptyset$, or equivalently $(\mathbf{p},\mathbf{T}) \in \mathcal{I}(\mathcal{P}_{S\Delta},\mathcal{T}_{S\Delta})$. Once proven, together with \eqref{form8} this implies \eqref{form7}.
%, because then we know that $(p,T) \mapsto (\mathbf{p},\mathbf{T})$ defines a $MN$-to-$1$ map from $\mathcal{I}(\mathcal{P}^{A}_{\Delta},\mathcal{T}^{A}_{\Delta})$ to $\mathcal{I}(\mathcal{P}_{S\Delta},\mathcal{T}_{S\Delta})$.

Proving that $B_{\mathbf{p}} \cap \mathbb{T}_{\mathbf{T}} \neq \emptyset$ is an exercise in using the triangle inequality: indeed $AB_{p} \cap A\mathbb{T}_{T} \neq \emptyset$ implies that $\dist(p,T) \lesssim A\Delta = cS\Delta$, and in particular $\dist(p,T) \leq S\Delta$ if $c > 0$ was chosen small enough. Recalling that $B_{\mathbf{p}}$ is a disc of diameter $10S\Delta$ around $\mathbf{p}$ and $\mathbb{T}_{T}$ is an ordinary tube of width $10S\Delta$ around $\mathbf{T}$, it follows that $B_{\mathbf{p}} \cap T_{\mathbf{T}} \neq \emptyset$.

Now, combining \eqref{form6}+\eqref{form7}, and the fact that $M/|\calP_{\Delta}| = 1/|\calP_{S\Delta}|$ and $N/|\calT_{\Delta}| = 1/|\calT_{S\Delta}|$, we may deduce that
\begin{align*} \iota(\Delta) = \frac{|\mathcal{I}(\mathcal{P}_{\Delta},\mathcal{T}_{\Delta})|}{\Delta|\mathcal{P}_{\Delta}||\mathcal{T}_{\Delta}|} & \stackrel{\eqref{form6}+\eqref{form7}}{\lesssim_{\mathbf{C},\eta}} \frac{\Delta^{-\eta^{2}}S^{-1}MN \cdot |\mathcal{I}(\mathcal{P}_{S\Delta},\mathcal{T}_{S\Delta})|}{\Delta|\mathcal{P}_{\Delta}||\mathcal{T}_{\Delta}|}\\
& = \Delta^{-\eta^{2}} \cdot \frac{|\mathcal{I}(\mathcal{P}_{S\Delta},\mathcal{T}_{S\Delta})|}{(S\Delta)|\mathcal{P}_{S\Delta}||\mathcal{T}_{S\Delta}|} = \Delta^{-\eta^{2}}\iota(S\Delta). \end{align*} 
This completes the proof of \eqref{form9}.

\subsubsection*{Conclusion of the proof} By hypothesis $\iota(\delta) \geq \delta^{-\kappa}$, and on the other hand trivially $\iota(1) \leq 1$. From these facts, and the monotonicity of the incidence quotients \eqref{form9} in the "low" cases, we may easily infer that the "high" case must occur at some scale $\Delta = \Delta_{j}$ with $j \in \{1,\ldots,n\}$. Indeed, if this were not the case, we could "chain" the inequalities \eqref{form9} for all scales $\Delta = \Delta_{j}$ with $j \in \{1,\ldots,n\}$ to deduce, for a suitable constant $\mathbf{c} = \mathbf{c}(\mathbf{C},\eta) > 0$, that
\begin{equation}\label{form10}
	1 \geq \iota(\Delta_{0})\ge (\mathbf{c}\delta^{\eta^{2}})^{n-1}\iota(\Delta_{n}) \geq \mathbf{c}^{1/\eta} \cdot \delta^{\eta^{2}/\eta -\kappa} \geq \mathbf{c}^{1/\eta} \cdot \delta^{-\kappa/2},
\end{equation}
recalling that $\eta \leq \kappa/2$. Provided that $\delta>0$ is small enough, we reach a contradiction. Therefore, the high case must occur for some index $j \in \{1,\ldots,n\}$, and we choose the largest index at which this happens: thus $\Delta := \Delta_{j}$ is the smallest "high" scale. Since the low case occurred at all larger indices, repeating the argument at \eqref{form10} shows that
\begin{equation}\label{eq:iotabig}
\frac{|\mathcal{I}(\mathcal{P}_{\Delta},\mathcal{T}_{\Delta})|}{\Delta |\mathcal{P}_{\Delta}||\mathcal{T}_{\Delta}|} = \iota(\Delta) \gtrsim_{\mathbf{C},\eta} \delta^{\eta -\kappa}
\end{equation}
On the other hand, since the "high" case occurs at scale $\Delta$, we have \eqref{form5} at our disposal:
\begin{displaymath} |\mathcal{I}(\mathcal{P}_{\Delta},\mathcal{T}_{\Delta})| \lesssim_{\eta} \sqrt{S\Delta^{-1}|\mathcal{P}_{\Delta}||\mathcal{T}_{\Delta}|} \stackrel{\eqref{eq:iotabig}}{\lesssim_{\mathbf{C},\eta}}\sqrt{\delta^{\kappa-2\eta}\Delta^{-2}|\mathcal{I}(\mathcal{P}_{\Delta},\mathcal{T}_{\Delta})|}. \end{displaymath}
Rearranging this yields $|\mathcal{I}(\mathcal{P}_{\Delta},\mathcal{T}_{\Delta})| \lesssim_{\mathbf{C},\eta} \delta^{\kappa-2\eta}\Delta^{-2}$, as claimed in \eqref{form3}. \end{proof}

\section{Application to sub-uniformly distributed sets}

In this section we use Proposition \ref{prop1} to prove Theorem \ref{thm:dim-estimate}. First, we need the following decomposition lemma: 

\begin{lemma}\label{lemma1}
	Let $\eta\in (0,1)$, and set $S\coloneqq\delta^{-\eta}$. Suppose that $\{\delta,S\delta,S^{2}\delta,\ldots,1\} = \{\Delta_{n},\ldots,\Delta_{0}\}\subset 2^{-\mathbb{N}}$, where $n= \eta^{-1}$. Given $\calP\subset\calD_\delta$, there is a partition
	\begin{equation*}
	\calP = \bigcup_{i=1}^N \calP_i
	\end{equation*}
	such that $\calP_{i}$ are pairwise disjoint, each $\calP_i$ is sub-$\{\Delta_j\}_{j=0}^n$-uniform with constant $2$, and $N\le (-Cn^{-1}\log\delta)^n$.
\end{lemma}

\begin{proof}
	Note that for any $j\in\{0,\dots,n-1\}$ each $\mathbf{p}\in\calP_{\Delta_{j}}$ contains at most $S^2$ cubes from $\calP_{\Delta_{j + 1}}$. For any $p\in\calP$, we will denote by $\mathbf{p}^j$ the dyadic $\Delta_j$-parent of $p$.
	
	For $0\le k\le \log_2 S^2$ we set
	\begin{equation*}
	\calP^k = \{p\in\calP\ :\ 2^{k-1}<|\mathbf{p}^{n - 1} \cap\calP|\le 2^k \}.
	\end{equation*}
	Clearly, the sets $\calP^k$ are sub-$\{\Delta_{n},\Delta_{n - 1},\Delta_{0}\}$-uniform.
	
	We proceed inductively. Suppose that $\widetilde{\calP}\subset\calP$ is sub-$\{\Delta_{n},\ldots,\Delta_{m},\Delta_{0}\}$-uniform, where $m \geq 2$. For $0\le k\le \log_2 S^2$ we set
	\begin{equation*}
	\widetilde{\calP}^k = \{p\in\widetilde{\calP}\ :\ 2^{k-1}<|\mathbf{p}^{m - 1}\cap\widetilde{\calP}_{\Delta_m}|\le 2^k \}.
	\end{equation*}
	It follows that the sets $\widetilde{\calP}^k$ are sub-$\{\Delta_{n},\ldots,\Delta_{m},\Delta_{m - 1},\Delta_{0}\}$-uniform
	
	After performing this subpartitioning procedure $n$ times, we end up with sub-$\{\Delta_j\}_{j=0}^n$-uniform sets. Each time, we increased the number of partitions at most by a factor of $\log_2 S^2+1\sim -\eta\log\delta$.
\end{proof}

%\begin{cor} Let $\mathcal{P} \subset \mathcal{D}_{\delta}$, and let $\mathcal{T} \subset \mathcal{T}^{\delta}$ be a family of tubes which are $\kappa$-heavy w.r.t. $\mathcal{P}$:
%	\begin{displaymath} |\{p \in \mathcal{P} : B_{p} \cap \mathbb{T}_{T} \neq \emptyset\}| \geq \delta^{1 - \kappa}|\mathcal{P}|, \qquad T \in \mathcal{T}. \end{displaymath} 
%	Then, there exists a disjoint decomposition $\mathcal{T} = \mathcal{T}^{1} \cup \ldots \cup \mathcal{T}^{N}$ with $N \leq (\log(\tfrac{1}{\delta}))^{n}$ such that $n \sim 1/\kappa$, and for each $j \in \{1,\ldots,N\}$ a scale $\Delta_{j} \in 2^{-\N} \cap [\delta,\delta^{\kappa/2}]$ such that
%	\begin{displaymath} |\mathcal{I}(\mathcal{P}_{\Delta_{j}},\mathcal{T}^{j}_{\Delta_{j}})| \leq \delta^{\kappa/2}\Delta_{j}^{-2}. \end{displaymath}
%\end{cor}

We use Proposition \ref{prop1} and the decomposition lemma above to get an estimate on incidences between a sub-uniform set $\calP$ and a family of tubes $\calT$ consisting exclusively of ``heavy tubes''.

\begin{lemma}\label{lem:uniformP}
	For every $\mathbf{C} > 0$, $\kappa \in (0,1)$ and $\eta \in (0,\tfrac{\kappa}{3}]$ the following holds for all $\delta \in 2^{-\N}$ small enough such that also $S := \delta^{-\eta} \in 2^{-\N}$.  Consider the scale sequence $\{\delta,S\delta,S^{2}\delta,\ldots,1\} = \{\Delta_{n},\Delta_{n - 1},\ldots,\Delta_{0}\}$. Suppose that $\mathcal{P} \subset \mathcal{D}_{\delta}$ is sub-$\{\Delta_j\}_{j=0}^n$-uniform with constant $\mathbf{C}$, and that $\mathcal{T} \subset \mathcal{T}^{\delta}$ satisfies
	\begin{equation}\label{eq:heavytubes}
	|\mathcal{I}(\calP,\{T\})| =  |\{p \in \mathcal{P} : B_{p} \cap \mathbb{T}_{T} \neq \emptyset\}| \geq \delta^{1 - \kappa}|\mathcal{P}|, \qquad T\in\mathcal{T}.
	\end{equation}
	Then,
	\begin{equation}\label{form16} \mathcal{H}_{\infty}^{2-\kappa+3\eta}(\mathcal{I}(\mathcal{P},\mathcal{T})) \lesssim_{\mathbf{C},\eta} \delta^{\eta/2}. \end{equation}
\end{lemma}
\begin{proof}
	We apply Lemma \ref{lemma1} to the tube family $\calT$ to obtain a partition of $\calT$ into sub-$\{\Delta_j\}_{j=0}^n$-uniform (with constant $2$) subfamilies $\calT_1,\dots,\calT_N$. It follows from \eqref{eq:heavytubes} that
	\begin{equation*}
	|\mathcal{I}(\calP,\mathcal{T}_i)|\ge\delta^{1-\kappa}|\calP||\calT_i|, \qquad i \in \{1,\ldots,N\}.
	\end{equation*}
	Thus, we may apply Proposition \ref{prop1} for each $1 \leq i \leq N$ to obtain $\Delta(i)\in\{\Delta_1,\dots,\Delta_n\}$ such that 
	\begin{equation}\label{form17}
	|\mathcal{I}(\mathcal{P}_{\Delta(i)},\mathcal{T}_{i,\Delta(i)})| \lesssim_{\mathbf{C},\eta} \delta^{\kappa-2\eta} \Delta(i)^{-2}.
	\end{equation}
	Since $\mathcal{I}(\calP,\calT) = \bigcup_{i = 1}^N\mathcal{I}(\calP,\calT_i)$, we note that
	\begin{equation*}
	\mathcal{I}(\calP,\calT)\subset \bigcup_{i = 1}^{N} (\mathcal{I}(\mathcal{P},\mathcal{T}_{i}))_{\Delta_{i}} \stackrel{\mathrm{L.\,} \ref{lemma2}}{\subset} \bigcup_{i = 1}^N \mathcal{I}(\mathcal{P}_{\Delta(i)},\mathcal{T}_{i,\Delta(i)}).
	\end{equation*}
	We may then use this covering to obtain \eqref{form16}:
	\begin{align*} \mathcal{H}_\infty^{2-\kappa+3\eta}(\mathcal{I}(\mathcal{P},\mathcal{T})) & \le \sum_{i=1}^N \Delta(i)^{2-\kappa+3\eta} |\mathcal{I}(\mathcal{P}_{\Delta(i)},\mathcal{T}_{i,\Delta(i)})|\\
	&\stackrel{\eqref{form17}}{\lesssim_{\mathbf{C},\eta}} \sum_{i=1}^N\Delta(i)^{-\kappa+3\eta}\delta^{\kappa-2\eta}\\
	&\le N\delta^{\kappa-2\eta}\delta^{-\kappa+3\eta}\lesssim\delta^{\eta/2},
	\end{align*}
	where in the last estimate we used $N\lesssim (\log\frac{1}{\delta})^{\eta^{-1}}\le\delta^{-\eta/2}$ for $\delta$ small enough.
\end{proof}
Recall from Definition \ref{def:subUniform} that a bounded set $K$ is sub-uniformly distributed if it satisfies $|K|_{R} \cdot |K \cap Q|_{r} \le C|K|_{r},$ for all $Q \in \mathcal{D}_{R}(K), \, 0 < r \leq R < \infty$. Below, $\underline{\dim}_{\mathrm{B}} K$ denotes the lower box-counting dimension of $K$.
\begin{lemma}\label{lem:dimhEqualdimbox}
	If $K$ is compact and sub-uniformly distributed, then $\dim_{\mathrm{H}} K=\underline{\dim}_{\mathrm{B}} K$.
\end{lemma}
\begin{proof}
	Since $\dim_{\mathrm{H}} K \le \underline{\dim}_{\mathrm{B}} K$ is always true, we only need to show $\dim_{\mathrm{H}} K \ge \underline{\dim}_{\mathrm{B}} K$. By the definition of $\underline{\dim}_{\mathrm{B}} K$, for every $\epsilon > 0$ there exists $r_{\epsilon} > 0$ such that $|K|_{r} \geq r^{-t + \epsilon}$ for all $r \in (0,r_{\epsilon}]$. Now, fix $\epsilon > 0$, and let $\mathcal{U}$ be an open cover of $K$ such that $\diam(U) \leq r_{\epsilon}$ for all $U \in \mathcal{U}$. We will show that $\sum_{U \in \mathcal{U}} \diam(U)^{t - \epsilon} \gtrsim 1$, where the constant only depends on the sub-uniformity constant of $K$. 
	
	Since $K$ is compact, there exists a finite sub-cover $\mathcal{U}_{0} \subset \mathcal{U}$. Let $\delta \in 2^{-\N}$ be so small that $\delta \leq \min\{\diam(U) : U \in \mathcal{U}_{0}\}$. For this $\delta$, by the sub-uniformity of $K$, we have
	\begin{displaymath} \frac{|K \cap U|_{\delta}}{|K|_{\delta}} \lesssim \frac{|K|_{\delta}}{|K|_{\diam(U)} \cdot |K|_{\delta}} \leq \diam(U)^{t - \epsilon}, \qquad U \in \mathcal{U}_{0}, \end{displaymath}
	using $\diam(U) \leq r_{\epsilon}$ in the second inequality. Further, choosing $\delta > 0$ smaller if necessary, we may assume that every $p \in \mathcal{D}_{\delta}(K)$ is contained in at least one element of $\mathcal{U}_{0}$. Then,
	\begin{displaymath} \sum_{U \in \mathcal{U}_{0}} \diam(U)^{t - \epsilon} \gtrsim \frac{1}{|K|_{\delta}} \sum_{U \in \mathcal{U}_{0}} |K \cap U|_{\delta} \geq 1, \end{displaymath} 
	as claimed. This shows that $\Hd K \geq t - \epsilon$, and letting $\epsilon \to 0$ completes the proof.  \end{proof}
Given $z\in \R^2$ and $e\in S^1$, recall that $\ell_{e,z} = z + \mathrm{span}(e)$. We consider the following modification of the heavy part $H(K,e,s)$ defined in the introduction:
\begin{equation*}
	\underline{H}(K,e,s) = \{z\in K\ : \ \underline{\dim}_{\mathrm{B}}(K\cap \ell_{e,z})\ge s\}.
\end{equation*}
We may also consider $\overline{H}(K,e,s)$, where $\underline{\dim}_{\mathrm{B}}(K\cap \ell_{e,z})$ is replaced by $\overline{\dim}_{\mathrm{B}}(K\cap \ell_{e,z})$. Since $\dim_{\mathrm{H}}(A)\le \underline{\dim}_{\mathrm{B}}(A)\le\overline{\dim}_{\mathrm{B}}(A)$ for every bounded set $A\subset\R^2$, it is clear that 
\begin{equation*}
	H(K,e,s)\subset \underline{H}(K,e,s)\subset \overline{H}(K,e,s).
\end{equation*}
Therefore, Theorem \ref{thm:dim-estimate} is a consequence of the following slightly stronger estimate.
\begin{proposition}\label{prop:Ahlfors}
	Let $K\subset \R^2$ be a compact sub-uniformly distributed set with $\dim_{\mathrm{H}} K =t\in (0,2)$. For every $\max\{t-1,0\}<s\le 1$ and a.e. $e\in S^1$
	\begin{equation}\label{eq:slicing}
		\dim_{\mathrm{H}} \underline{H}(K,e,s) \le \max\{t-s,0\}.
	\end{equation}
	Additionally, if $K$ is $t$-Ahlfors regular, then $\dim_{\mathrm{H}} \overline{H}(K,e,s) \le \max\{t-s,0\}$.
\end{proposition}
\begin{proof}[Proof]
	We assume that $s\le t$, otherwise the problem is trivial. We assume also $K\subset B(1)$. Fix $0<\eta  < \tfrac{1}{6}(s-t+1)$.
	
	Given $\delta\in 2^{-\mathbb{N}}$, we say that $T \in \mathcal{T}^{\delta}$ is roughly parallel to $e\in S^1$ if $T$ contains some line parallel to $e$. For each $e\in S^1$ denote by $\mathcal{H}_{e,\delta}$ the family of all $T \in \mathcal{T}^{\delta}$ roughly parallel to $e$ satisfying
	\begin{equation}\label{eq:1}
	|\mathcal{I}(K_{\delta}, \{T\})| \ge \delta^{t-s+2\eta}|K|_{\delta}.
	\end{equation}
	Set
	\begin{equation*}
	\calH_\delta = \bigcup_{e\in S^1} \calH_{e,\delta}.
	\end{equation*}
	Note that the tubes $T \in\calH_\delta$ are heavy in the sense of Lemma \ref{lem:uniformP}: more precisely, according to \eqref{eq:1}, they satisfy \eqref{eq:heavytubes} with $\kappa \coloneqq s-t+1-2\eta\in (0,1)$ and $\calP=K_{\delta}$.
	
	Recall that for each $z\in \underline{H}(K,e,s)$ we have $\underline{\dim}_{\mathrm{B}}(K\cap \ell_{e,z})\ge s$. It follows that for all sufficiently small $\delta \in 2^{-\N}$ there exists a tube $T\in\calT^\delta$ containing $\ell_{e,z}$ and satisfying
	\begin{equation}\label{eq:bla}
		|\mathcal{I}(K_\delta, \{T\})| \ge \delta^{-s+\eta}.
	\end{equation}
	(Here we committed a small inaccuracy: the dyadic tubes have currently been defined so that the slopes of their core lines lie between $0$ and $1$. So, to find any dyadic $\delta$-tubes containing $\ell_{e,z}$ at all, we need to constrain $e$ to an arc $J \subset S^{1}$ corresponding to these slopes. We leave it to the reader to check that it suffices to prove \eqref{eq:slicing} for a.e. $e \in J$.) 
	
	Let $\delta_k\to 0$ be a sequence of dyadic numbers such that $|K|_{\delta_k}\le \delta_k^{-t-\eta}$ for all $k \in \N$ (such a sequence exists by Lemma \ref{lem:dimhEqualdimbox}). Then, for $T \in \mathcal{T}^{\delta_{k}}$ satisfying \eqref{eq:bla} we have
	\begin{equation}\label{eq:2}
	|\mathcal{I}(K_{\delta_k}, \{T\})| \ge \delta_k^{-s+\eta}\ge  \delta_k^{t-s+2\eta}|K|_{\delta_k},
	\end{equation}
	and in particular, $T\in \calH_{e,\delta_k}$. This shows that for every $z\in \underline{H}(K,e,s)$ there exists $k_0 = k_0(e,z)$ such that for all $k\ge k_0$ we have $z\in K\cap \bigcup_{T\in\calH_{e,\delta_k}}T$. Hence, for any $N\in\mathbb{N}$
	\begin{equation}\label{eq:3}
	\underline{H}(K,e,s) \subset K\cap \bigcup_{k\ge N}\bigcup_{T\in\calH_{e,\delta_k}}T.
	\end{equation}
	For brevity of notation, we set $\calK_k\coloneqq\calK_{\delta_k}=\mathcal{D}_{\delta_{k}}(K),$ $\calH_{e,k}\coloneqq\calH_{e,\delta_k}$, $\calH_k\coloneqq\calH_{\delta_k}$, and $H_{e,k}\coloneqq K\cap \bigcup_{T\in\calH_{e,k}}T.$
	
	Observe that, since $K$ is sub-uniformly distributed, the family $\mathcal{K}_{k}$ is sub-$\{\Delta_j\}_{j=0}^n$-uniform (with some constant $\mathbf{C}\ge 1$) for any sequence of scales $\{\Delta_j\}_{j=0}^n$ larger than $\delta_k$. In particular, Lemma \ref{lem:uniformP} tells us that for all $\delta_k\le\delta_0=\delta_0(\kappa,\eta,\mathbf{C})$, where $\eta$ and $\kappa$ are as above, we have the following bound for the incidences between the heavy tubes $\calT=\calH_k$ and $\calP= \mathcal{K}_{k}$:
	\begin{equation}\label{eq:bla2}
		\mathcal{H}_{\infty}^{2-\kappa+3\eta}(\mathcal{I}(\mathcal{K}_{k},\mathcal{H}_k)) \lesssim_{\eta,\mathbf{C}} \delta_k^{\eta/2},
	\end{equation}
	From now on we assume that $\delta_k\le\delta_0$ for all $k\in\mathbb{N}$.
	
%	For each $e\in S^1$ let
%	\begin{equation*}
%		{H}_{e,k} = K\cap \bigcup_{T\in\mathcal{T}_{e,k}}T.
%	\end{equation*}
%	By the discussion around \eqref{eq:bla} we have $\underline{H}(K,e,s)\subset \bigcup_{k\ge N}H_{e,k}$ for any $N\in\mathbb{N}$.
	
%	Observe that for any $e\in S^1$ we have
%	\begin{equation*}
%		\mathcal{H}_{\infty}^{1-\kappa+\epsilon_0}(H_{e,k})\ d\theta \lesssim \mathcal{H}_{\infty}^{1-\kappa+\epsilon_0}(\mathcal{I}(K_{\delta_k},\mathcal{T}_{e,k})).
%	\end{equation*}
	Let $\calR=\{R_i\}\in\mathcal{D}(\R^4)$ be a family of dyadic cubes covering $\mathcal{I}(\mathcal{K}_{k},\mathcal{H}_k)\subset\R^4$ which attains the bound \eqref{eq:bla2}. Each $R_i$ can be written as $R_i=p_i\otimes T_i = p_i\times p_{T_i}$, with $p_i, p_{T_i}\in \calD(\R^2)$. For $e\in S^1$ let $\calR_e\subset\calR$ be the family of cubes $R_i$ for which $T_i$ is roughly parallel to $e$, and set
	\begin{equation*}
		\calP_e = \{p_i \ :\ R_i = p_{i} \otimes T_{i} \in \calR_e\}.
	\end{equation*}
	We claim that $\calP_e$ is a covering of $H_{e,k}$. Indeed, note that $\calR_e$ are a covering of the set of incidences $\mathcal{I}(\calK_k, \calH_{e,k})$. At the same time, for every $z\in H_{e,k}$ we have that the cube $p\in\calK_k$ containing $z$ is incident to some heavy tube $T\in\calH_{e,k}$. Thus, $p\otimes T\in \mathcal{I}(\calK_k, \calH_{e,k}),$ and there exists some $R_i=p_i\otimes T_i\in \calR_e$ such that $p\otimes T \subset p_i\otimes T_i$, and in particular $p\subset p_i$.
	
	Using the fact that for any $R_i\in\calR$
	\begin{equation*}
		\mathcal{H}^1(\{e\in S^1 \ :\ R_i\in\calR_e\}) \lesssim \ell(R_i),
	\end{equation*}
	and the fact that $\calP_e$ covers $H_{e,k}$, we get
	\begin{align*}
		\int_{S^1} \mathcal{H}_{\infty}^{1-\kappa+3\eta}(H_{e,k})\, d\mathcal{H}^{1}(e) & \lesssim \int_{S^1} \sum_{p_i\in \calP_e} \ell(p_i)^{1-\kappa+3\eta}\, d\mathcal{H}^{1}(e)\\
		& \lesssim \int_{S^1} \sum_{R_i\in \calR_e} \ell(R_i)^{1-\kappa+3\eta}\, d\mathcal{H}^{1}(e)\\
		& = \sum_{R_i\in\calR}\ell(R_i)^{1-\kappa+3\eta}\cdot \mathcal{H}^1(\{e\in S^1 \ :\ R\in\calR_e\})\\
		& \lesssim \sum_{R_i\in\calR}\ell(R_i)^{2-\kappa+\epsilon_0} \lesssim \mathcal{H}_{\infty}^{2-\kappa+3\eta}(\mathcal{I}(\mathcal{K}_{k},\mathcal{H}_k)) \stackrel{\eqref{eq:bla2}}{\lesssim_{\eta,\mathbf{C}}} \delta_k^{\eta/2}.
	\end{align*}
	Recalling that by \eqref{eq:3} for any $N\in\mathbb{N}$ we have $\underline{H}(K,e,s)\subset \bigcup_{k\ge N}H_{e,k}$, we arrive at
	\begin{align*}
		\int_{S^1} \mathcal{H}_{\infty}^{1-\kappa+3\eta}(\underline{H}(K,e,s))\, d\mathcal{H}^{1}(e) & \leq \liminf_{N \to \infty} \sum_{k \geq N} \int_{S^1} \mathcal{H}_{\infty}^{1-\kappa+3\eta}(H_{e,k})\, d\mathcal{H}^{1}(e)\\
		& \lesssim_{\eta,\mathbf{C}} \liminf_{N \to \infty} \sum_{k\ge N}\delta_k^{\eta/2} = 0.
	\end{align*}
	Thus, for a.e. $e\in S^1$ we have $\mathcal{H}_{\infty}^{1-\kappa+3\eta}(\underline{H}(K,e,s))=0$. Since $1-\kappa+3\eta = t-s+5\eta$, letting $\eta\to 0$ gives \eqref{eq:slicing}.
	
	In the case of Ahlfors regular $K$, it is straightforward to modify the proof above to show the estimate for $\overline{H}(K,e,s)$ instead of $\underline{H}(K,e,s)$. In this case we have $|K|_{\delta}\lesssim \delta^{-t}$ for all $\delta$, and not just for some sequence $\delta_k\to 0$. At the same time, for every $z\in \overline{H}(K,e,s)$ we have a dyadic sequence $\delta_{k}(z)\to 0$ such that every $\delta_{k}(z)$-tube containing $\ell_{e,z}$ satisfies
	\begin{equation*}
		|\mathcal{I}(K_{\delta_{k}(z)}, \{T\})| \ge \delta_{k}(z)^{-s+\eta}\ge \delta_k(z)^{t-s+2\eta}|K|_{\delta_{k}(z)}.
	\end{equation*}
	This means that $T\in \calH_{e,\delta_k(z)}$. Hence, for every $N\in\mathbb{N}$
	\begin{equation*}
		\overline{H}(K,e,s)\subset K\cap\bigcup_{\delta\le 2^{-N}}\bigcup_{T\in\calH_{e,\delta}}T,
	\end{equation*}
	which corresponds to \eqref{eq:3}. After that, we may proceed exactly as before.
\end{proof}

\section{Continuous sharpness examples}

\begin{thm}\label{mainExample} For every $t \in (1,2]$ and $s \in [t - 1,\tfrac{1}{3}(2t - 1)]$ there exists a compact set $K \subset \R^{2}$ with the following properties. 
\begin{itemize}
\item $\Hd K = t$.
\item For $e \in S^{1}$, let $\mathcal{L}_{e}$ be the family of lines $\ell \subset \R^{2}$ parallel to $e$ such that $\Hd (K \cap \ell) \geq s$. Then $\Hd (K \cap (\cup \mathcal{L}_{e})) = \min\{2t - 3s,t\}$ for all $e \in S^{1}$.
\end{itemize} 
\end{thm}

In fact, our construction gives the following variant of Theorem \ref{mainExample}, from which Theorem \ref{mainExample} as stated follows by point-line duality. If $\mathcal{L}$ is a family of lines, and $A \subset \R^{2}$ is a set, we write $\mathcal{L}(A) := \{\ell \in \mathcal{L} : A \cap \ell \neq \emptyset\}$. If $A = \{z\}$, we abbreviate $\mathcal{L}(\{z\}) =: \mathcal{L}(z)$.

\begin{thm}\label{mainExampleDual} For every $t \in (1,2]$ and $s \in [t - 1,\tfrac{1}{3}(2t - 1)]$ there exists a compact line set $\mathcal{L} \subset \mathcal{A}(2,1)$ with the following properties.
\begin{itemize}
\item $\Hd \mathcal{L} = t$.
\item For $x \in [0,1]$, let 
\begin{displaymath} R_{x} := \{z \in \{x\} \times \R : \Hd \mathcal{L}(z) \geq s\}. \end{displaymath}
Then, $\Hd \mathcal{L}(R_{x}) = \min\{2t - 3s,t\}$.
\end{itemize}

\end{thm}

\subsection{The building block}\label{s:buildingBlock} The basic building block of our construction uses the notion of $(\delta,s)$-sets.
\begin{definition}[$(\delta,s,C)$-set]\label{def:deltaSSet} For $\delta \in 2^{-\N}$, $s \in [0,d]$, and $C > 0$, a non-empty bounded set $P \subset \R^{d}$ is called a \emph{$(\delta,s,C)$-set} if
	\begin{displaymath} |P \cap B(x,r)|_{r} \leq Cr^{s}|P|_{\delta}, \qquad x \in \R^{d}, \, r \in [\delta,1]. \end{displaymath}
	If $\mathcal{P}$ is a finite union of dyadic cubes (possibly of different side-lengths), we say that $\mathcal{P}$ is a $(\delta,s,C)$-set if the union $\cup \mathcal{P}$ is a $(\delta,s,C)$-set in the sense above. \end{definition}
It is useful to note that if $P$ is a  $(\delta,s,C)$-set, then $|P|_{\delta} \geq \delta^{-s}/C$. This follows by applying the defining inequality with $r := \delta$ and to any $B(x,r)$ intersecting $P$.

Let $\tau := \min\{2t - 3s,t\}$. The proof of Theorem \ref{mainExampleDual} is based on the existence of sets $\mathcal{P}_{\Delta} \subset \mathcal{D}_{\Delta}$, $\Delta \in 2^{-\N}$, with the following properties \nref{P1}-\nref{P2}.
\begin{itemize}
\item[(P1) \phantomsection \label{P1}] If $x \in [0,1]$, then $\mathcal{P}_{\Delta} \cap (\{x\} \times \R)$ contains a non-empty $(\Delta,\tau - 1,C)$-set. Note that $\tau - 1 \geq 0$ by the hypothesis $s < \tfrac{1}{3}(2t - 1)$. If $\tau = 1$, we are merely claiming here that $\mathcal{P}_{\Delta} \cap (\{x\} \times \R) \neq \emptyset$ for all $x \in [0,1]$.
\item[(P2) \phantomsection \label{P2}] There exists a $\Delta^{s}$-separated set $\mathcal{E}_{\Delta} \subset \mathcal{D}_{\Delta}(S^{1})$ such that $\Delta^{-s} \lesssim |\mathcal{E}_{\Delta}| \leq \Delta^{-s}$, and
\begin{displaymath} |\pi_{e}(\mathcal{P}_{\Delta})|_{\Delta} \lesssim \Delta^{-(s + \tau)/2}, \qquad e \in \cup \mathcal{E}_{\Delta}. \end{displaymath}
\end{itemize}
As a sanity check, note that $(s + \tau)/2 \leq (2t - 2s)/2 \leq 1$ since $s \geq t - 1$ by assumption. It may also be fun to know that $\mathcal{P}_{\Delta}$ is a $(\Delta,\tau,C)$-set, but this will not be explicitly needed. Further, we will need the following "quasi nested" property of the sets $\mathcal{E}_{\Delta}$:
\begin{itemize}
\item[(E) \phantomsection \label{E}] If the sequence $\{\bar{\Delta}_{n}\}_{n \in \N} \subset 2^{-\N}$ decays so rapidly that $\bar{\Delta}_{n} < \bar{\Delta}_{n - 1}^{n}$, then the set
\begin{equation}\label{defE} E := \bigcap_{n \in \N} (\cup \mathcal{E}_{\bar{\Delta}_{n}}) \end{equation}
satisfies $\Hd E = s$.
\end{itemize}
We suspect that the existence of $\mathcal{P}_{\Delta}$ and $\mathcal{E}_{\Delta}$ is "well-known", but since a precise reference was difficult to come by, we give the full details in Appendix \ref{s:constructBlock}.

\subsection{The general idea}\label{s:generalIdea} To prove Theorem \ref{mainExampleDual}, we will construct three objects simultaneously: a compact (Cantor-type) set $R \subset [0,1]^{2}$, and two compact line sets $\mathcal{L}_{F}$ and $\mathcal{L}_{G}$. These objects will satisfy the following properties:
\begin{itemize}
\item $\Hd \mathcal{L}_{F} \leq (3s + \tau)/2 \leq t$ and $\Hd \mathcal{L}_{G} \leq \tau \leq t$ (here still $\tau = \min\{2t - 3s,t\}$).
\item $\Hd \mathcal{L}_{F}(z) \geq s$ for all $z \in R$.
\item $\Hd \mathcal{L}_{G}(R \cap (\{x\} \times \R)) = \tau$ for all $x \in [0,1]$.
\end{itemize}
Once this has been accomplished, we define $\mathcal{L} := \mathcal{L}_{F} \cup \mathcal{L}_{G}$. Then $\Hd \mathcal{L} \leq t$,
\begin{displaymath} \Hd \mathcal{L}(z) \geq \Hd \mathcal{L}_{F}(z) \geq s, \qquad z \in R, \end{displaymath}
and
\begin{displaymath} \Hd \mathcal{L}(R_{x}) \geq \Hd \mathcal{L}_{G}(R \cap (\{x\} \times \R)) = \min\{2t - 3s,t\}, \qquad x \in [0,1]. \end{displaymath}
This will complete the proof of Theorem \ref{mainExampleDual}.

\begin{remark} Since $\mathcal{L} = \mathcal{L}_{F} \cup \mathcal{L}_{G}$, it may first seem that we can deduce the stronger bound $\Hd \mathcal{L} \leq \max\{(3s + \tau)/2,\tau\}$. However, the right hand side equals "$t$" in both possible cases $\tau = 2t - 3s$ and $\tau = t$.  \end{remark}

The construction of all the objects $R,\mathcal{L}_{F},\mathcal{L}_{G}$ will be based on a fixed but very rapidly decreasing "double" scale sequence $\{\Delta_{n},\delta_{n}\}$ of the form
\begin{displaymath} 1 =: \Delta_{0} \gg \delta_{0} \gg \Delta_{1} \gg \delta_{1} \gg \ldots > 0, \end{displaymath}
where $\delta_{n},\Delta_{n} \in 2^{-\N}$. It will always be crucial to choose $\delta_{n}$ much smaller than $\Delta_{n}$, and also $\Delta_{n + 1}$ much smaller than $\delta_{n}$. In fact, the only requirements will be that $\delta_{n} < \Delta_{n}^{2n}$ and $\Delta_{n + 1} < \delta_{n}^{2(n + 1)}$ for all $n \geq 0$. We now fix a sequence $\{\Delta_{n},\delta_{n}\}$ with these properties.

\subsection{The set $R$} We define the set $R \subset [0,1]^{2}$ by the following iterative procedure. We will have
\begin{displaymath} R = \bigcap_{n = 0}^{\infty} R_{\Delta_{n}}, \end{displaymath}
where $R_{\Delta_{n}}$ is the union of a finite family $\mathcal{R}_{\Delta_{n}}$ of closed dyadic $\Delta_{n}$-squares. Write $\mathcal{R}_{\Delta_{0}} := \{[0,1]^{2}\}$ and $R_{0} := [0,1]^{2}$. Assume that $R_{\Delta_{n}},\mathcal{R}_{\Delta_{n}}$ have already been constructed. Fix $\delta_{n} \in 2^{-\N}$ with $\delta_{n} < \Delta_{n}$, and let
\begin{equation}\label{form24} \mathcal{R}_{\delta_{n}} := \{Q \in \mathcal{D}_{\delta_{n}} : Q \subset R_{\Delta_n}\}. \end{equation}
Thus $\mathcal{R}_{\delta_{1}} = \mathcal{D}_{\delta_{1}}$. (In this section, the notation $\mathcal{D}_{\delta}$ will refer to closed dyadic sub-squares of $[0,1]^{2}$.) Next, fix $\Delta_{n + 1} \ll \delta_{n}$, and consider the "building block"
\begin{displaymath} \mathcal{P}_{n} := \mathcal{P}_{\Delta_{n + 1}/\delta_{n}}. \end{displaymath}
For each square $Q \in \mathcal{R}_{\delta_{n}} \subset \mathcal{D}_{\delta_{n}}$, let $\mathcal{P}_{Q}$ be a copy of $\mathcal{P}_{n}$ which has been rescaled by $\delta_{n}$ and then translated inside $Q$, thus 
\begin{displaymath} \mathcal{P}_{Q} := S_{Q}(\mathcal{P}_{n}) \subset \mathcal{D}_{\Delta_{n + 1}}(Q), \end{displaymath}
where $S_{Q}$ is the homothety taking $[0,1]^{2}$ to $Q$. We then define
\begin{equation}\label{defR} \mathcal{R}_{\Delta_{n + 1}} := \bigcup_{Q \in \mathcal{R}_{\delta_{n}}} \mathcal{P}_{Q} \subset \mathcal{D}_{\Delta_{n + 1}} \quad \text{and} \quad R_{\Delta_{n + 1}} := \cup \mathcal{R}_{\Delta_{n + 1}}. \end{equation}
This completes the definitions of the families $\mathcal{R}_{\Delta_{n}},\mathcal{R}_{\delta_{n}}$, and the set $R$.

Figure \ref{fig4} shows how the set $R_{\Delta_{n + 1}}$ might look inside a single square $Q \in \mathcal{R}_{\delta_{n - 1}}$. Note that first $Q$ is replaced by the "$\tau$-dimensional" set $\mathcal{P}_{\Delta_{n}/\delta_{n - 1}}$. Next there follows a period of "$2$-dimensional branching" between the scales $\Delta_{n} > \delta_{n}$, and finally another period of "$\tau$-dimensional branching" between the scales $\delta_{n} > \Delta_{n + 1}$.
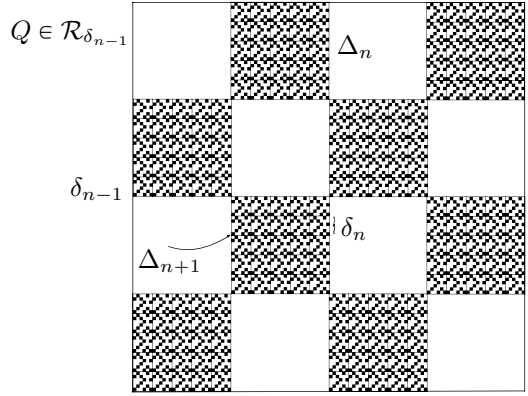
\begin{figure}[h!]
\begin{center}
\begin{overpic}[scale = 0.8]{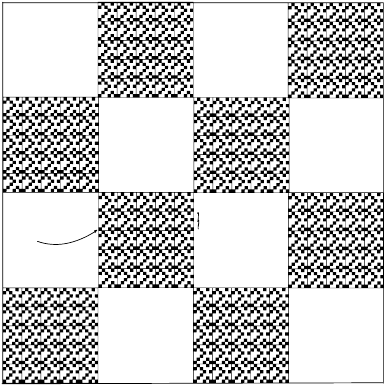}
\put(-15,50){\small{$\delta_{n - 1}$}}
\put(-30,90){\small{$Q \in \mathcal{R}_{\delta_{n - 1}}$}}
\put(52,86){\small{$\Delta_{n}$}}
\put(2,32){\small{$\Delta_{n + 1}$}}
\put(53,40){\small{$\delta_{n}$}}
\end{overpic}
\caption{The set $Q \cap R_{\Delta_{n + 1}}$ for $Q \in \mathcal{R}_{\delta_{n - 1}}$.}\label{fig4}
\end{center}
\end{figure}

\subsection{The line family $\mathcal{L}_{F}$} We will next define the line family $\mathcal{L}_{F}$ announced in Section \ref{s:generalIdea}. We will also verify all the stated properties of $\mathcal{L}_{F}$ immediately. The idea is that the set $R$ constructed above has small projections in many directions. The line family $\mathcal{L}_{F}$ will be defined as the family of all pre-image lines of all of these small projections.

\begin{lemma}\label{lemma4} Assuming that $\Delta_{n} < \delta_{n - 1}^{2n}$ for all $n \geq 1$, we have

\begin{displaymath} |\pi_{e}(R_{\Delta_{n}})|_{\Delta_{n}} \leq \Delta_{n}^{-(s + \tau)/2 - 1/n}, \qquad e \in \cup \mathcal{E}_{\Delta_{n}/\delta_{n - 1}}, \, n \geq 1. \end{displaymath} 
\end{lemma}

\begin{proof} Fix $n \geq 1$. Note that $\pi_{e}(R_{\Delta_{n}})$ equals the union of the projections $\pi_{e}(\mathcal{P}_{Q})$ for $Q \in \mathcal{R}_{\delta_{n - 1}}$. On the other hand, by \nref{P2} we have
\begin{displaymath} |\pi_{e}(\mathcal{P}_{Q})|_{\Delta_{n}} = |\pi_{e}(\mathcal{P}_{\Delta_{n}/\delta_{n - 1}})|_{\Delta_{n}/\delta_{n - 1}} \leq \Delta_{n}^{-(s + \tau)/2}, \quad e \in \cup \mathcal{E}_{\Delta_{n}/\delta_{n - 1}}.  \end{displaymath} 
Since $\delta_{n - 1}^{-2} \leq \Delta_{n}^{-1/n}$ by assumption, it follows from the above, and $|\mathcal{R}_{\delta_{n - 1}}| \leq \delta_{n - 1}^{-2}$, that
\begin{displaymath} |\pi_{e}(R_{\Delta_{n}})|_{\Delta_{n}} \leq \delta_{n - 1}^{-2} \cdot \Delta_{n}^{-(s + \tau)/2} \leq \Delta_{n}^{-(s + \tau)/2 - 1/n}, \qquad e \in \cup \mathcal{E}_{\Delta_{n}/\delta_{n - 1}}. \end{displaymath}
This completes the proof of the lemma. \end{proof}

We now define
\begin{displaymath} E := \bigcap_{n = 1}^{\infty} (\cup \mathcal{E}_{\Delta_{n}/\delta_{n - 1}}) \subset S^{1}, \end{displaymath}
and the line family 
\begin{displaymath} \mathcal{L}_{F} := \bigcup_{e \in E} \mathcal{L}_{F}(e) := \bigcup_{e \in E} \{\pi_{e}^{-1}\{\pi_{e}(z)\} : z \in R\}. \end{displaymath} 
\begin{proposition} Assuming that $\Delta_{n} \leq \delta_{n - 1}^{2n}$ for $n \geq 1$, we have
\begin{displaymath} \Hd \mathcal{L}_{F} \leq (3s + \tau)/2. \end{displaymath} 
\end{proposition} 

\begin{proof} Fix $n \geq 1$ and recall from \nref{P2} that
\begin{displaymath} |E|_{\Delta_{n}} \leq \delta_{n - 1}^{-1} \cdot |\mathcal{E}_{\Delta_{n}/\delta_{n - 1}}| \lesssim \delta_{n - 1}^{-1} \cdot (\Delta_{n}/\delta_{n - 1})^{-s} \leq \Delta_{n}^{-s - 1/n}, \qquad n \in \N. \end{displaymath} 
On the other hand, for each $e \in E \subset \cup \mathcal{E}_{\Delta_{n}/\delta_{n - 1}}$, the lines of $\mathcal{L}_{F}(e)$ can be covered by $\leq \Delta_{n}^{-(s + \tau)/2 - 1/n}$ tubes of width $\Delta_{n}$ by Lemma \ref{lemma4} (and one can use the same $\Delta_{n}$-tubes to cover all lines $\mathcal{L}_F(e)$ whose angular components are within $\leq \Delta_{n}$). Consequently $|\mathcal{L}_{F}|_{\Delta_{n}} \leq \Delta_{n}^{-(3s + \tau)/2 - 2/n}$. This implies that $\Hd \mathcal{L}_{F} \leq \underline{\dim}_{\mathrm{B}} \mathcal{L}_{F} \leq (3s + \tau)/2$. \end{proof}

The following proposition is immediate from the definitions:
\begin{proposition} We have
\begin{displaymath} \Hd \mathcal{L}_{F}(z) \geq s, \qquad z \in R. \end{displaymath} \end{proposition} 

\begin{proof} Evidently $\Hd \mathcal{L}_{F}(z) \geq \Hd E$. Moreover $\Hd E = s$ by \nref{E}, provided that the ratios $\bar{\Delta}_{n} = \Delta_{n}/\delta_{n - 1}$ decay so rapidly that $\bar{\Delta}_{n} < \bar{\Delta}_{n - 1}^{n}$, or equivalently
\begin{displaymath} \Delta_{n} < \Delta_{n - 1}^{n} \left(\frac{\delta_{n - 1}}{\delta_{n - 2}}\right)^{n}, \qquad n \geq 1. \end{displaymath}
This follows from our assumptions: $\Delta_{n} < \delta_{n - 1}^{2n} \leq \Delta_{n - 1}^{n}\delta_{n - 1}^{n}$. \end{proof} 

\subsection{The line family $\mathcal{L}_{G}$} We then define the line family $\mathcal{L}_{G}$ announced in Section \ref{s:generalIdea}. We restate here the required properties of $\mathcal{L}_{G}$:
\begin{itemize}
\item[(G1) \phantomsection \label{G1}] $\Hd \mathcal{L}_{G} = \tau = \min\{2t - 3s,t\}$.
\item[(G2) \phantomsection \label{G2}] $\Hd \mathcal{L}_{G}(R_{x}) = \tau$ for all $x \in [0,1]$. Here and below $R_{x} := R \cap (\{x\} \times \R)$.
\end{itemize}

A slightly informal description of $\mathcal{L}_{G}$ is the following: $\mathcal{L}_{G}$ is a uniform set with $\tau$-dimensional branching between scales $\Delta_{n} > \delta_{n}$ and $2$-dimensional branching between scales $\delta_{n} > \Delta_{n + 1}$. Therefore $\mathcal{L}_{G}$ has exactly the opposite features as the set $R$, which instead had $2$-dimensional branching between scales $\Delta_{n} > \delta_{n}$ and $\tau$-dimensional branching between scales $\delta_{n} > \Delta_{n + 1}$.

We now give a more precise definition. We define $\mathcal{L}_{G}$ as the set of lines which are contained in the intersection of the following nested sequence of dyadic tube families (recall Notation \ref{not1}). First, recall that $\Delta_{0} = 1$, and let $\mathcal{T}_{\Delta_{0}} \subset \mathcal{T}^{4}$ consist of all the dyadic $4$-tubes intersecting $[0,1]^{2}$. 

Next, assume that $\mathcal{T}_{\Delta_{n}} \subset \mathcal{T}^{4\Delta_{n}}$ has already been constructed for some $n \geq 0$. For every $\mathbf{T} \in \mathcal{T}_{\Delta_{n}}$ we define $\mathcal{T}_{\delta_{n}}(\mathbf{T}) \subset \mathcal{T}^{4\delta_{n}}$ to be a maximally separated set of dyadic $(4\delta_{n})$-tubes contained in $\mathbf{T}$ of cardinality $|\mathcal{T}_{\delta_{n}}(\mathbf{T})| = (\Delta_{n}/\delta_{n})^{\tau}$. In fact, the property of $\mathcal{T}_{\delta_{n}}(\mathbf{T})$ we really need is this: if the tubes in $\mathcal{T}_{\delta_{n}}(\mathbf{T})$ are rescaled by $(4\Delta_{n})^{-1}$, and the resulting family (in the parameter space $[0,1]^{2}$) is restricted to a square $Q \subset [0,1]^{2}$ of side-length $\tfrac{1}{100}$, then the remaining family of dyadic $(\delta_{n}/\Delta_{n})$-tubes is a $(\delta_{n}/\Delta_{n},\tau,C)$-set for an absolute constant $C > 0$. So, informally speaking, we need $\mathcal{T}_{\delta_{n}}(\mathbf{T})$ to be a moderately well-distributed $(\delta_{n}/\Delta_{n},\tau)$-set of tubes, modulo rescaling by $(4\Delta_{n})^{-1}$.

Then, we set
\begin{equation}\label{form35} \mathcal{T}_{\delta_{n}} := \bigcup_{\mathbf{T} \in \mathcal{T}_{\Delta_{n}}} \mathcal{T}_{\delta_{n}}(\mathbf{T}). \end{equation}
We note that $|\mathcal{T}_{\delta_{n}}| \leq \Delta_{n}^{-2} \cdot \delta_{n}^{-\tau} \leq \delta_{n}^{-\tau - 1/n}$ for $n \geq 0$, recalling that $\delta_{n} \leq \Delta_{n}^{2n}$.

 Finally, for each $T \in \mathcal{T}_{\delta_{n}}$, we define $\mathcal{T}_{\Delta_{n + 1}}(T)$ to consist of all the dyadic $(4\Delta_{n + 1})$-tubes contained in $T$, and we set
\begin{equation}\label{form36} \mathcal{T}_{\Delta_{n + 1}} := \bigcup_{T \in \mathcal{T}_{\delta_{n}}} \mathcal{T}_{\Delta_{n + 1}}(T). \end{equation}
We then define
\begin{displaymath} \mathcal{L}_{G} := \bigcap_{n \in \N} (\cup \mathcal{T}_{\delta_{n}}) = \bigcap_{n \in \N} (\cup \mathcal{T}_{\Delta_{n}}). \end{displaymath}
Property \nref{G1} is rather clear by construction:
\begin{proposition} $\Hd \mathcal{L}_{G} \leq \tau$. \end{proposition} 

\begin{proof} Note that all the lines in $\mathcal{L}_{G}$ can be covered by the dyadic tubes in $\mathcal{T}_{\delta_{n}}$, $n \geq 0$. Since $|\mathcal{T}_{\delta_{n}}| \leq \delta_{n}^{-\tau - 1/n}$ for $n \geq 0$, it follows that $\Hd \mathcal{L}_{G} \leq \underline{\dim}_{\mathrm{B}} \mathcal{L} \leq \tau$. \end{proof}

To understand the subsets $\mathcal{L}_{G}(R_{x}) \subset \mathcal{L}_{G}$, and to prove \nref{G2}, we start by setting up some notation. We write $L_{x} := \{x\} \times \R$, and identifying $L_{x} \cong \R$, further
\begin{displaymath} \mathcal{I}_{\delta_{n}}(x) := \{I \in \mathcal{D}_{\delta_{n}}(L_{x}) : I \subset R_{\delta_{n}}\} \quad \text{and} \quad \mathcal{I}_{\Delta_{n}}(x) := \{I \in \mathcal{D}_{\Delta_{n}}(L_{x}) : I \subset R_{\Delta_{n}}\}. \end{displaymath}
We agree that the dyadic intervals above are closed (and recall also that $R_{\delta_{n}},R_{\Delta_{n}}$ are defined as unions of closed dyadic squares).

For $\mathbf{I} \in \mathcal{I}_{\Delta_{n}}(x)$, let $\mathcal{I}_{\delta_{n}}(\mathbf{I}) := \{I \in \mathcal{I}_{\delta_{n}}(x) : I \subset \mathbf{I}\}$ and similarly for $I \in \mathcal{I}_{\delta_{n}}(x)$, let $\mathcal{I}_{\Delta_{n + 1}}(I) := \{\mathbf{I} \in \mathcal{I}_{\Delta_{n + 1}}(x) : \mathbf{I} \subset \mathcal{I}\}$ (we suppress "$x$" from the notation for simplicity). Now,
\begin{equation}\label{form33} \mathcal{I}_{\delta_{n}}(\mathbf{I}) = \mathcal{D}_{\delta_{n}}(\mathbf{I}), \qquad \mathbf{I} \in \mathcal{I}_{\Delta_{n}}(x), \end{equation}
This is because the set $\mathcal{R}_{\delta_{n}}$ consisted of all the $\delta_{n}$-squares contained in $R_{\Delta_{n}}$.

By similar reasoning,
\begin{equation}\label{form37} |\mathcal{I}_{\Delta_{n + 1}}(I)| \sim (\delta_{n}/\Delta_{n + 1})^{1 - \tau}, \qquad I \in \mathcal{I}_{\delta_{n}}(x). \end{equation}
This is a direct consequence of \nref{P1}, and the definition of $\mathcal{R}_{\Delta_{n + 1}}$ as a union of the sets $\mathcal{P}_{Q}$. In fact, we can more precisely say that the $\delta_{n}^{-1}$-dilation of the family $\mathcal{I}_{\Delta_{n + 1}}(I)$ contains a non-empty $(\Delta_{n + 1}/\delta_{n},\tau - 1,C)$-set for all $I \in \mathcal{I}_{\delta_{n}}$. As a minor technical point, we will denote by $\mathcal{I}_{\Delta_{n + 1}}'(I)$ the subset of $\mathcal{I}_{\Delta_{n + 1}}(I)$ whose $\delta_{n}^{-1}$-dilation is a $(\Delta_{n + 1}/\delta_{n},\tau - 1,C)$-set (then \eqref{form37} remains true for $\mathcal{I}_{\Delta_{n + 1}}'(I)$).

We then arrive at the key property \nref{G2} of the line family $\mathcal{L}_{G}$.

\begin{proposition}\label{prop2} We have
\begin{displaymath} \Hd \mathcal{L}_{G}(R_{x}) = \tau, \qquad x \in [0,1]. \end{displaymath}
\end{proposition} 

The proof will be based on the following lemma:
\begin{lemma}\label{lemma6} Let $d \geq 1$ and $t \in [0,d]$. Let $E \subset [0,1]^{d}$ be a Cantor set of the form
\begin{displaymath} E = \bigcap_{m = 0}^{\infty} E_{m}, \end{displaymath}
where each $E_{m}$ is a union of closed dyadic $\delta_{m}$-cubes, and $\{\delta_{m}\}_{m = 0}^{\infty} \subset 2^{-\N}$ be a super-geometrically decaying sequence: thus $\delta_{m} \leq \epsilon \delta_{m - 1}$ for all $m \geq m_{\epsilon}$.

Let $C \geq 1$. For $Q \in \mathcal{D}_{\delta_{m}}(E)$, assume that $E \cap Q$ is a \emph{relative} $(\delta_{m + 1},u,C)$-subset of $Q$. By this, we mean that the rescaled set $S_{Q}(E \cap Q) \subset [0,1]^{2}$ is a non-empty $(\delta_{m + 1}/\delta_{m},u,C)$-set, or equivalently
\begin{displaymath} |E \cap \mathbf{q}|_{\delta_{m + 1}} \leq C\left(\tfrac{r}{\delta_{m}} \right)^{u} \cdot |E \cap Q|_{\delta_{m + 1}}, \qquad \mathbf{q} \in \mathcal{D}_{r}(Q), \, \delta_{m + 1} \leq r \leq \delta_{m}. \end{displaymath}
Then $\Hd E \geq u$. \end{lemma}

\begin{proof} Without loss of generality, we may assume that $\delta_0=1$ and $E_0 = [0,1]^d$. We define a measure $\mu$ on $E$ in the obvious way, requiring $\mu([0,1]^{d}) = 1$, and then
\begin{displaymath} \mu(Q) := \frac{\mu(Q^{(n - 1)})}{|E \cap Q^{(n - 1)}|_{\delta_{n}}}, \qquad Q \in \mathcal{D}_{\delta_{n}}(E), \, n \geq 1, \end{displaymath}
where $Q^{(n - 1)} \in \mathcal{D}_{\delta_{n - 1}}(E)$ is the unique $\delta_{n - 1}$-cube containing $Q$. Iterating the definition, and applying the fact that the $\rho$-covering number of a non-empty $(\rho,t,C)$-set is at least $\rho^{-t}/C$, we find
\begin{displaymath} \mu(Q) = \prod_{j = 1}^{n} |E \cap Q^{(j - 1)}|_{\delta_{j}}^{-1} \leq C^{n} \prod_{j = 1}^{n} \left(\tfrac{\delta_{j}}{\delta_{j - 1}}\right)^{u} = C^{n}\delta_{n}^{u}, \qquad Q \in \mathcal{D}_{\delta_{n}}(E).  \end{displaymath}
This would roughly show that $\mu$ satisfies a $(u - \epsilon)$-dimensional Frostman condition for radii $r \in \{\delta_{n}\}_{n \in \N}$. To treat the intermediate radii $\delta_{n} \leq r \leq \delta_{n - 1}$, we need to apply the relative $(\delta_{n},u,C)$-set property of the sets $E \cap Q$ in a stronger way than above. Fix $\delta_{n} \leq r \leq \delta_{n - 1}$, and let $\mathbf{Q} \in \mathcal{D}_{r}(E)$. Let $\mathbf{Q}^{(n - 1)} \in \mathcal{D}_{\delta_{n - 1}}(E)$ be the unique $\delta_{n - 1}$-cube containing $\mathbf{Q}$. Then,
\begin{displaymath} \mu(\mathbf{Q}) = |\mathbf{Q} \cap E|_{\delta_{n}} \cdot \frac{\mu(\mathbf{Q}^{(n - 1)})}{|E \cap \mathbf{Q}^{(n - 1)}|_{\delta_{n}}} \leq C\left(\tfrac{r}{\delta_{n - 1}}\right)^{u} \cdot \mu(\mathbf{Q}^{(n - 1)}) \leq C^{n+1}r^{u}. \end{displaymath} 
Now, it remains to note that $C^{n+1} \leq \delta_{n - 1}^{-\epsilon} \leq r^{-\epsilon}$ for all $n \geq n_{\epsilon}$ by the super-geometric decay of $\{\delta_{n}\}_{n \in \N}$. Therefore $\mu(B(x,r)) \lesssim_{\epsilon} r^{u - \epsilon}$ for all $\epsilon > 0$ and all $0<r\le r_\epsilon$, and consequently $\Hd E \geq u$.  \end{proof}

\begin{proof}[Proof of Proposition \ref{prop2}] We prove the proposition by constructing a $\tau$-dimensional Cantor-type subset $\mathcal{E}(x) \subset \mathcal{L}_{G}(R_{x})$ which satisfies the hypotheses of Lemma \ref{lemma6} (applied to dyadic tubes instead of dyadic cubes).

Define $\mathcal{T}_{\Delta_{0}}(x)$ to consist of all the elements $\mathbf{T} \in \mathcal{T}_{\Delta_{0}}$ such that
\begin{displaymath} \mathbf{I}_{0} := [0,1] \subset \mathbf{T} \cap L_{x}. \end{displaymath}
Assume next that $\mathcal{T}_{\Delta_{n}}(x)$ has already been defined for some $n \geq 0$, with the property that whenever $\mathbf{T} \in \mathcal{T}_{\Delta_{n}}(x)$, then $\mathbf{T} \cap L_{x}$ contains an interval $\mathbf{I}_{\mathbf{T}} \in \mathcal{I}_{\Delta_{n}}(x)$.

Fix $\mathbf{T} \in \mathcal{T}_{\Delta_{n}}(x)$. We define $\mathcal{T}_{\delta_{n}}(\mathbf{T},x)$ to consist of all those tubes $T \in \mathcal{T}_{\delta_{n}}$ which are contained in $\mathbf{T}$, and such that $T \cap L_{x}$ contains an interval $I_{T} \in \mathcal{I}_{\delta_{n}}(x)$ with $I_{T} \subset \mathbf{I}_{\mathbf{T}}$. We remark that $T \cap L_{x}$ may contain more than one interval from $\mathcal{I}_{\delta_{n}}(x)$ with this property, and we simply choose one of them and denote it by $I_{T}$. At the same time, there are at most $\lesssim 1$ possible choices, since all dyadic tubes in $\calT^{\delta_n}$ form an angle $\gtrsim 1$ with $L_x$, which means that $\mathcal{H}^1({T}\cap L_x)\lesssim \delta_n$.

We set
\begin{displaymath} \mathcal{T}_{\delta_{n}}(x) := \bigcup_{\mathbf{T} \in \mathcal{T}_{\Delta_{n}}(x)} \mathcal{T}_{\delta_{n}}(\mathbf{T},x). \end{displaymath}
It remains to define $\mathcal{T}_{\Delta_{n + 1}}$. Fix $T \in \mathcal{T}_{\delta_{n}}(x)$. We define $\mathcal{T}_{\Delta_{n + 1}}(T,x)$ to consist of all those tubes $\mathbf{T} \in \mathcal{T}_{\Delta_{n + 1}}$ which are contained in $T$, and such that $\mathbf{T} \cap L_{x}$ contains an interval $\mathbf{I}_{\mathbf{T}} \in \mathcal{I}_{\Delta_{n + 1}}'(x)$ with $\mathbf{I}_{\mathbf{T}} \subset I_{T}$. We then set
\begin{displaymath} \mathcal{T}_{\Delta_{n + 1}}(x) := \bigcup_{T \in \mathcal{T}_{\delta_{n}}(x)} \mathcal{T}_{\Delta_{n + 1}}(T,x). \end{displaymath}
We are then prepared to define the Cantor set $\mathcal{E}(x) \subset \mathcal{L}_{G}(R_{x})$, whose dimension will equal $\tau$. We say that $\ell \in \mathcal{E}(x)$ if there exists a sequence of tubes $\mathbf{T}_{1},T_{1},\mathbf{T}_{2},T_{2},\ldots$ such that $\mathbf{T}_{n} \in \mathcal{T}_{\Delta_{n}}(x)$, $T_{n} \in \mathcal{T}_{\delta_{n}}(x)$, and $\mathbf{T}_{n} \subset T_{n - 1} \subset \mathbf{T}_{n - 1}$ for all $n \geq 1$, and
\begin{displaymath} \ell = \bigcap_{n = 1}^{\infty} \mathbf{T}_{n} = \bigcap_{n = 1}^{\infty} T_{n}. \end{displaymath}
We record that $\ell \in \mathcal{L}_{G}(R_{x})$. The reason is that the tubes $T_{n},\mathbf{T}_{n}$ come with associated intervals $I_{n} \in \mathcal{I}_{\delta_{n}}(x)$ and $\mathbf{I}_{n} \in \mathcal{I}_{\Delta_{n}}'(x)$ (as in the construction above) with the property that $\mathbf{I}_{n} \subset \mathbf{T}_{n} \cap L_{x}$. As $n \to \infty$, the intervals $\mathbf{I}_{n}$ (or $I_{n}$) converge to a unique point $z \in R_{x}$, and therefore $\ell \in \mathcal{L}_{G}(\{z\}) \subset \mathcal{L}_{G}(R_{x})$.

\begin{figure}[h!]
\begin{center}
\begin{overpic}[scale = 0.5]{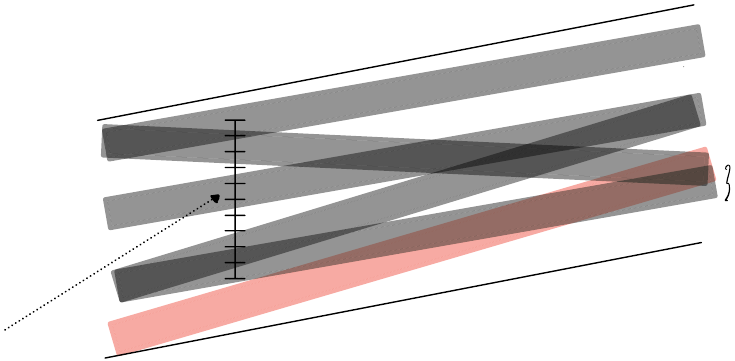}
\put(-3,1){$\mathbf{I}_{\mathbf{T}}$}
\put(101,23){$\delta_{n}$}
\put(13,35){$\mathbf{T}$}
\end{overpic}
\caption{The grey tubes are in the family $\mathcal{T}_{\delta_{n}}(\mathbf{T},x)$, but the red tube is not.}\label{fig5}
\end{center}
\end{figure}

It remains to show that $\Hd \mathcal{E}(x) = \tau$. This work divides into two claims:
\begin{claim}  For all $\mathbf{T} \in \mathcal{T}_{\Delta_{n}}(x)$, $n \geq 0$, the family $\mathcal{T}_{\delta_{n}}(\mathbf{T},x) \subset \mathcal{T}^{\delta_{n}}$ is a relative $(\delta_{n},\tau,C)$-subset of $\mathbf{T}$ with an absolute constant $C > 0$.\end{claim} 

\begin{proof} Fix $\mathbf{T} \in \mathcal{T}_{\Delta_{n}}(x)$, $n \geq 0$. The family $\mathcal{T}_{\delta_{n}}(\mathbf{T},x)$ is depicted in Figure \ref{fig5}. Recall that by the definition of $\mathbf{T} \in \mathcal{T}_{\Delta_{n}}(x)$, there exists an interval $\mathbf{I}_{\mathbf{T}} \in \mathcal{I}_{\Delta_{n}}'(x)$ (also shown in Figure \ref{fig5}) such that $\mathbf{I}_{\mathbf{T}} \subset \mathbf{T} \cap L_{x}$. Note that $\mathbf{T}$ is a dyadic $(4\Delta_{n})$-tube, so the width of $\mathbf{T}$ is $4$ times the length of $\mathbf{I}_{\mathbf{T}}$.

 Next, recall that the entire family $\mathcal{T}_{\delta_{n}}(\mathbf{T})$ (defined above \eqref{form35}) consists of a maximally separated set of $\delta_{n}$-tubes contained in $\mathbf{T}$ and satisfying $|\mathcal{T}_{\delta_{n}}(\mathbf{T})| = (\Delta_{n}/\delta_{n})^{\tau}$. However, $\mathcal{T}_{\delta_{n}}(\mathbf{T},x)$ only consists of those $T \in \mathcal{T}_{\delta_{n}}(\mathbf{T})$ with the  additional property that there exists an interval $I_{T} \in \mathcal{I}_{\delta_{n}}(\mathbf{I}_{\mathbf{T}})$ such that $I_{T} \subset T \cap L_{x}$ and $I_{T} \subset \mathbf{I}_{\mathbf{T}}$. We claim that (at least) all the tubes $T \in \mathcal{T}_{\delta_{n}}(\mathbf{T})$ with 
 \begin{equation}\label{form34} T \cap L_{x} \subset \mathbf{I}_{\mathbf{T}} \end{equation}
 have this additional property. Indeed, the width of these tubes if $4\delta_{n}$, so $T \cap L_{x}$ contains some sub-interval $I_{T} \in \mathcal{D}_{\delta_{n}}(\mathbf{I}_{\mathbf{T}})$. However, recall from \eqref{form33} that $\mathcal{I}_{\delta_{n}}(\mathbf{I}_{\mathbf{T}}) = \mathcal{D}_{\delta_{n}}(\mathbf{I}_{\mathbf{T}})$, so in fact $I_{T} \in \mathcal{I}_{\delta_{n}}(\mathbf{I}_{\mathbf{T}})$, as desired.
 
 \begin{figure}[h!]
\begin{center}
\begin{overpic}[scale = 0.5]{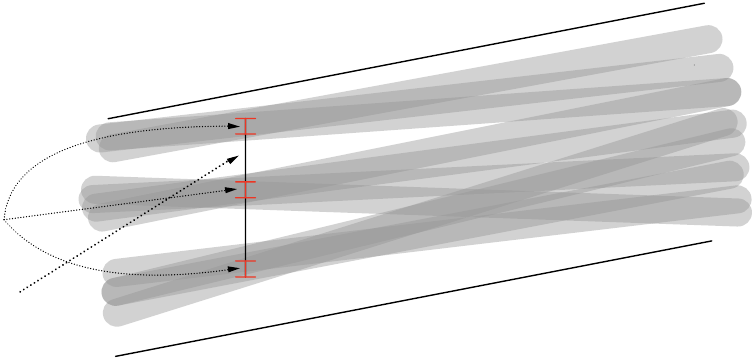}
\put(-2,4){$I_{T}$}
\put(101,23){$\delta_{n}$}
\put(13,35){$T$}
\put(-35,17){\small{$\mathbf{I} \in \mathcal{I}_{\Delta_{n + 1}}(I_{T})$}}
\end{overpic}
\caption{The grey tubes are in the family $\mathcal{T}_{\Delta_{n + 1}}(T,x)$.}\label{fig6}
\end{center}
\end{figure}
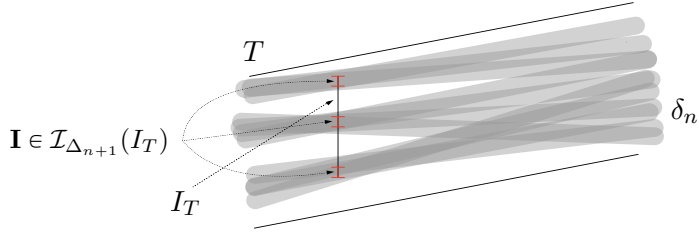

Now, since the tubes of $\mathcal{T}_{\delta_{n}}(\mathbf{T})$ are fairly well-distributed inside $\mathbf{T}$ (as discussed above \eqref{form35}), a positive absolute fraction of them satisfies \eqref{form34}. It follows that $\mathcal{T}_{\delta_{n}}(\mathbf{T},x)$ is a relative $(\delta_{n},\tau,C)$-subset of $\mathbf{T}$, with an absolute constant $C > 0$. \end{proof}

\begin{claim}  For all $T \in \mathcal{T}_{\delta_{n}}(x)$, $n \geq 0$, the family $\mathcal{T}_{\Delta_{n + 1}}(T,x) \subset \mathcal{T}^{\Delta_{n + 1}}$ is a relative $(\delta_{n},\tau,C)$-subset of $T$ with an absolute constant $C > 0$.\end{claim} 

\begin{proof} Fix $T \in \mathcal{T}_{\delta_{n}}(x)$. The family $\mathcal{T}_{\Delta_{n}}(T,x)$ is depicted in Figure \ref{fig6}. Recall that by the definition of $T \in \mathcal{T}_{\delta_{n}}(x)$, there exists an interval $I_{T} \in \mathcal{I}_{\delta_{n}}(x)$ (also shown in Figure \ref{fig6}) such that $I_{T} \subset T \cap L_{x}$. As before, the width of $T$ is $4$ times the length of $I_{T}$.

Next, recall that the entire family $\mathcal{T}_{\Delta_{n + 1}}(T)$ (defined above \eqref{form36}) consists of \textbf{all} the dyadic $(4\Delta_{n + 1})$-sub-tubes of $T$, or in other words
\begin{displaymath} \mathcal{T}_{\Delta_{n + 1}}(T) = \mathcal{T}^{4\Delta_{n + 1}}(T). \end{displaymath}
However, not (even nearly of) all of these tubes belong to $\mathcal{I}_{\Delta_{n + 1}}(T,x)$. Namely, the condition for $\mathbf{T} \in \mathcal{T}_{\Delta_{n + 1}}(T,x)$ is that there exists an interval $\mathbf{I} \in \mathcal{I}_{\Delta_{n + 1}}'(I_{T})$ such that $\mathbf{I} \subset \mathbf{T} \cap L_{x}$.

Recall from below \eqref{form37} that $\mathcal{I}_{\Delta_{n + 1}}'(I_{T})$ is a relative $(\Delta_{n + 1},\tau - 1,C)$-subset of $I_{T}$. Now, it remains to observe that the subset of $\mathcal{T}^{4\Delta_{n + 1}}(T) = \mathcal{T}_{\Delta_{n + 1}}(T)$ of tubes containing at least one interval $\mathbf{I} \in \mathcal{I}_{\Delta_{n + 1}}(I_{T})$ is a $\tau$-dimensional set, or more precisely a relative $(\Delta_{n + 1},\tau,C)$-subset of $T$. We omit the details: this tube family is essentially a product of a $(\tau - 1)$-dimensional set with an interval.   \end{proof} 

The previous two claims together show that the Cantor set $\mathcal{E}(x)$ of lines is the intersection of "a nested sequence of $\tau$-dimensional sets", and satisfies the hypotheses of Lemma \ref{lemma6}. Consequently $\Hd \mathcal{L}_{G}(R_{x}) \geq \Hd \mathcal{E}(x) \geq \tau$, as claimed in \nref{G2}. \end{proof}

We have now verified that the sets $R,\mathcal{L}_{F},\mathcal{L}_{G}$ satisfy the properties stated in Section \ref{s:generalIdea}, so the proof of Theorem \ref{mainExampleDual} is complete.

\appendix

\section{Construction of the building block}\label{s:constructBlock}

The purpose of this section is to provide a construction for the "building block" sets $\mathcal{P}_{\Delta}$ described in Section \ref{s:buildingBlock}. We repeat the properties here for the reader's convenience:
\subsection{The building block} Fix $\tau \in [1,2]$ and $s \in [0,2 - \tau]$ for this section. (These restrictions on $s,\tau$ are valid in Section \ref{s:buildingBlock} by the assumption in Theorems \ref{mainExample}-\ref{mainExampleDual} that $t \in (1,2]$ and $s \in [t - 1,\tfrac{1}{3}(2t - 1)]$.) We claim that for all $\Delta \in 2^{-\N}$ sufficiently small, there exists a set $\mathcal{P}_{\Delta} \subset \mathcal{D}_{\Delta}$ with the following properties \nref{P1}-\nref{P2}.
\begin{itemize}
\item[(P1)] If $x \in [0,1]$, then $\mathcal{P}_{\Delta} \cap (\{x\} \times \R)$ contains a non-empty $(\Delta,\tau - 1,C)$-set. If $\tau = 1$, this just means that $\mathcal{P}_{\Delta} \cap (\{x\} \times \R) \neq \emptyset$ for all $x \in [0,1]$.
\item[(P2)] There exists a $\Delta^{s}$-separated set $\mathcal{E}_{\Delta} \subset \mathcal{D}_{\Delta}(S^{1})$ satisfying $|\mathcal{E}_{\Delta}| \sim \Delta^{-s}$ and
\begin{displaymath} |\pi_{e}(\mathcal{P}_{\Delta})|_{\Delta} \lesssim \Delta^{-(s + \tau)/2}, \qquad e \in \cup \mathcal{E}_{\Delta}. \end{displaymath}
\end{itemize}
Additionally, we will need the following "quasi nested" property of the sets $\mathcal{E}_{\Delta}$:
\begin{itemize}
\item[(E)] If the sequence $\{\Delta_{n}\}_{n \in \N} \subset 2^{-\N}$ decays so rapidly that $\Delta_{n} < \Delta_{n - 1}^{n}$, then the set
\begin{equation}\label{form41} E := \bigcap_{n \in \N} (\cup \mathcal{E}_{\Delta_{n}}) \end{equation}
satisfies $\Hd E = s$.
\end{itemize}
We will only give the details in the case $\tau \in (1,2)$. The case $\tau = 2$ is trivial (then $s = 0$ and we can take $\mathcal{P}_{\Delta} = \mathcal{D}_{\Delta}$). The case $\tau = 1$ requires minor modifications which we leave to the reader. So, we fix $\tau \in (1,2)$ and $\Delta \in 2^{-\N}$, and for convenience we assume that $\{\Delta^{-\tau/2},\Delta^{\tau/2 - 1}\} \subset \N$. We will suppress $\Delta$ in our notation from now on. We start with the auxiliary set
\begin{displaymath} P' := \left\{\left(\Delta^{\tau/2}k,\Delta^{\tau/2}l \right) : 0 \leq k,l \leq \Delta^{-\tau/2} - 1 \right\} \subset [0,1)^{2}. \end{displaymath} 
By the assumption $\Delta^{\tau/2 - 1} \in \N$, we have $\Delta^{\tau/2} = \Delta n$ for some $n \in \N$. Therefore $P' \subset (\Delta \cdot \N)^{2}$, and further
\begin{displaymath} \mathcal{P}' := P' + [0,\Delta)^{2} \subset \mathcal{D}_{\Delta}. \end{displaymath}
The set $\mathcal{P}'$ would satisfy property (P2) (we will return to this later), but it severely fails (P1), being a product set. To fix this, we need to rotate $\mathcal{P}'$ slightly, as in Figure \ref{fig7}.
\begin{figure}[h!]
\begin{center}
\begin{overpic}[scale = 0.6]{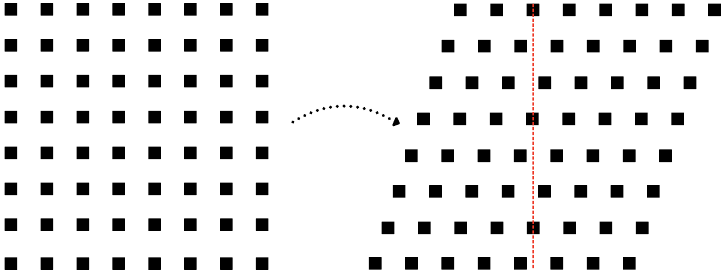}
\put(45,25){$R_{\theta}$}
\end{overpic}
\caption{The sets $\mathcal{P}'$ and $\mathcal{P}$.}\label{fig7}
\end{center}
\end{figure}

 Set $\theta := \Delta^{1 - \tau/2}$, and consider 
\begin{displaymath} R_{\theta}(x,y) := (x + \theta y,y) \quad \text{and} \quad P := R_{\theta}(P'). \end{displaymath}
We finally define $\mathcal{P} := P + [0,\Delta)^{2}$. Note that if $(x,y) = (\Delta^{\tau/2}k,\Delta^{\tau/2}l) \in P'$, then
\begin{displaymath} R_{\theta}(x,y) = (\Delta^{\tau/2}k + \Delta l,\Delta^{\tau/2}l) \subset (\Delta \cdot \N)^{2}, \end{displaymath}
so $\mathcal{P} \subset \mathcal{D}_{\Delta}(\R^{2})$. Some of the squares in $\mathcal{P}$ lie slightly outside $[0,1]^{2}$, and we simply discard those squares from $\mathcal{P}$ (only a tiny fraction is discarded, since $\theta = o_{\Delta \to 0}(1)$). 

We plan to verify (P1). The idea is roughly to show that if $(x,y) \in P$, then also the vertical translates $(x,y) + m(0,\Delta^{\tau - 1}) \in P$ for $m \in \{0,\ldots,\tfrac{1}{2}\Delta^{1 - \tau}\}$. Write $v := (0,\Delta^{\tau - 1})$.

\begin{claim} Assume that $x \in (\Delta \Z) \cap [\tfrac{1}{2},1]$, and $\Delta > 0$ is sufficiently small in terms of $\tau - 1 > 0$. Then there exists $y \in [0,\tfrac{1}{2})$ such that $(x,y) + mv \in P$ for $m \in \{0,\ldots,\tfrac{1}{2}\Delta^{1 - \tau}\}$. \end{claim} 

\begin{proof} Fix $x \in (\Delta \Z) \cap [\tfrac{1}{2},1]$. Then we may express $x$ in the form $x = \Delta^{\tau/2}k_{x} + \Delta l_{x}$ for some $\tfrac{1}{2}\Delta^{-\tau/2} \leq k_{x} \leq \Delta^{-\tau/2} - 1$ and $0 \leq l_{x} \leq \Delta^{\tau/2 - 1} \leq \tfrac{1}{2}\Delta^{-\tau/2} - 1$, using $\tau > 1$. In particular,
\begin{displaymath} (x,y) := (\Delta^{\tau/2}k_{x} + \Delta l_{x},\Delta^{\tau/2}l_{x}) \in P. \end{displaymath}
To proceed, we view $P$ as the image of the grid $\mathbb{G} = \{0,\ldots,\Delta^{-\tau/2} - 1\}^{2}$ under the map $\iota(k,l) = (\Delta^{\tau/2}k + \Delta l,\Delta^{\tau/2}l)$. It is easy to check that
\begin{displaymath} \iota(k,l) + mv = \iota(k - m,l + m\Delta^{\tau/2 - 1}), \qquad (k,l) \in \Z^{2}, \, m \in \Z. \end{displaymath}
Consequently, since $(x,y) = \iota(k_{x},l_{x}) \in P$, also $(x,y + mv) = \iota(k_{x},l_{x}) + mv \in P$ as long as
\begin{displaymath} (k_{x} - m,l_{x} + m\Delta^{\tau/2 - 1}) \in \mathbb{G}. \end{displaymath}
It remains to note that this is the case for $m \in \{0,\ldots,\tfrac{1}{2}\Delta^{1 - \tau}\}$. First, since $k_{x} \geq \tfrac{1}{2}\Delta^{-\tau/2}$, we have $k_{x} - m \geq 0$ for all $m \in \{0,\ldots,\tfrac{1}{2}\Delta^{1 - \tau}\}$ (noting that $\Delta^{1 - \tau} \leq \Delta^{-\tau/2}$). Similarly, since $l_{x} \leq \tfrac{1}{2}\Delta^{-\tau/2} - 1$, we have 
\begin{displaymath} l_{x} + m\Delta^{\tau/2 - 1} \leq (\tfrac{1}{2}\Delta^{-\tau/2} - 1) + \tfrac{1}{2}\Delta^{-\tau/2} = \Delta^{-\tau/2} - 1. \end{displaymath}
for all $m \in \{0,\ldots,\tfrac{1}{2}\Delta^{1 - \tau}\}$. Thus $(k_{x} - m,\ell_{x} + m\Delta^{\tau/2 - 1}) \in \mathbb{G}$ for $m \in \{0,\ldots,\tfrac{1}{2}\Delta^{-\tau/2}\}$, as claimed.  \end{proof}

The previous claim implies that if $x \in [\tfrac{1}{2},1]$, then $\mathcal{P} \cap (\{x\} \times \R)$ contains a $(\Delta,\tau - 1,C)$-set for an absolute constant $C > 0$. In Property (P1) we desired the same for all $x \in [0,1]$, but this problem can be fixed by replacing $\mathcal{P}$ by $\mathcal{P} \cup [\mathcal{P} - (\tfrac{1}{2},0)]$. This substitution has no impact on property (P2), which we verify next.

For notational convenience, we will here parametrise orthogonal projections as 
\begin{displaymath} \pi_{\lambda}(x,y) = x + \lambda y, \qquad \lambda \in \R. \end{displaymath} 
With this notation
\begin{displaymath} \pi_{\lambda}(R_{\theta}(x,y)) = (x + \theta y) + \lambda y = \pi_{\lambda + \theta}(x,y), \qquad (x,y) \in \R^{2}. \end{displaymath}
In particular, $\pi_{\lambda}(P) = \pi_{\lambda + \theta}(P')$ for all $\lambda \in \R$, where we recall that $\theta = \Delta^{1 - \tau/2} \in [0,1]$.

We then consider the following "direction set"
\begin{equation}\label{defEApp} \mathcal{E} = \mathcal{E}_{\Delta} = \left\{\lambda - \theta : \lambda = \tfrac{p}{q}, \, 0 \leq p, q \leq \Delta^{-s/2} \right\}. \end{equation}
We remark at this point that clearly $|\mathcal{E}| \lesssim \Delta^{-s}$.
\begin{claim}\label{c1} If $e \in \mathcal{E}$, or even $\dist(e,\mathcal{E}) \leq \Delta$, then $|\pi_{e}(\mathcal{P})|_{\Delta} \lesssim \Delta^{-(s + \tau)/2}$. \end{claim} 

\begin{proof} Since $\mathcal{P} = P + [0,\Delta)^{2}$, it suffices to show that $|\pi_{e}(P)| \lesssim \Delta^{-(s + \tau)/2}$. Fix $e = \lambda - \theta \in \mathcal{E}$. Then $\pi_{e}(P) = \pi_{\lambda}(P')$, so it suffices to show that $|\pi_{\lambda}(P')| \lesssim \Delta^{-(s + \tau)/2}$.

Write $\lambda = \tfrac{p}{q}$ as in the definition of $\mathcal{E}$, then fix $r \in \pi_{\lambda}(P')$, and let $(x,y) = \Delta^{\tau/2}(k,l) \in P'$ with $\pi_{\lambda}(x,y) = r$. Notice that also
\begin{displaymath} \pi_{\lambda}((x,y) + \Delta^{\tau/2}(m p,-mq)) = r + \Delta^{\tau/2}m \cdot (p - \tfrac{p}{q} \cdot q) = r, \qquad m \in \Z. \end{displaymath} 
Since $p,q \leq \Delta^{-s/2}$, we have $mp,mq \leq \Delta^{-\tau/2}$ for all $m \in \N \cap [0,\Delta^{(s - \tau)/2}]$. For such values of "$m$", we observe that
\begin{displaymath} (x,y) + \Delta^{\tau/2}(mp,-mq) \in \Delta^{\tau/2} \cdot \{-\Delta^{-\tau/2},\ldots,\Delta^{-\tau/2}\}^{2} =: \mathbb{H}. \end{displaymath} 
We have shown that $|\pi_{\lambda}^{-1}\{r\} \cap \mathbb{H}| \geq \Delta^{(s - \tau)/2}$. Since this was true for all $r \in \pi_{\lambda}(P')$, we deduce that $|\pi_{\lambda}(P')| \leq \Delta^{(\tau - s)/2} |\mathbb{H}| \sim \Delta^{-(s + \tau)/2}$, as claimed. \end{proof}

To complete the proof of property (P2), it remains to show that $\mathcal{E}$ contains a $\Delta^{s}$-separated subset of cardinality $\sim \Delta^{-s}$. We prove the following stronger claim:

\begin{claim}\label{dirichletClaim} Let $I \subset [0,1]$ be an interval with $|I| \geq C\Delta^{s/2}\log(1/\Delta)$ for a sufficiently large constant $C > 0$. Then $|\mathcal{E} \cap I|_{\Delta^{s}} \sim |I|\Delta^{-s}$. In particular $|\mathcal{E}|_{\Delta^{s}} \sim \Delta^{-s}$. \end{claim}

\begin{remark} This claim is no doubt well-known, and follows from the equidistribution properties of the \emph{Farey sequence}. We nevertheless give the brief details. \end{remark}

\begin{proof}[Proof of Claim \ref{dirichletClaim}] Recall Dirichlet's theorem on diophantine approximation: for every $x \in [0,1]$ and $N := \Delta^{-s/2} \in \N$, there exist $p,q \in \Z$ such that $1 \leq q \leq \Delta^{-s/2}$ and
\begin{equation}\label{form39} \left|x - \frac{p}{q}\right| \leq \frac{\Delta^{s/2}}{q}. \end{equation}
In particular this is true for $x \in \tfrac{1}{2}I$, where $I \subset [0,1]$ is the interval from the statement. In this case it follows from \eqref{form39} that additionally $\dist(p,q(\tfrac{1}{2}I)) \leq \Delta^{s/2} \leq 1$. Therefore $p \in \Z$ must lie in an interval $I_{q} \supset q(\tfrac{1}{2}I)$ of length $|I_{q}| \leq q|I| + 1$.

We rewrite the conclusion of Dirichlet's theorem as the inclusion
\begin{displaymath} \tfrac{1}{2}I \subset \bigcup_{q = 1}^{\Delta^{-s/2}} \bigcup_{p \in I_{q} \cap \Z} B\left(\tfrac{p}{q},\tfrac{\Delta^{s/2}}{q}\right). \end{displaymath}
Let us consider the question: how much of $\tfrac{1}{2}I$ can be covered by that part of the previous union where $q \leq c\Delta^{-s/2}$? Here $c > 0$ is a suitable constant to be determined in a moment. By sub-additivity, and taking into account that $\card(I_{q} \cap \Z) \leq 2q|I| + 2$, the measure of this "bad" part $I_{\mathrm{bad}} \subset \tfrac{1}{2}I$ is at most
\begin{displaymath} |I_{\mathrm{bad}}| \leq \sum_{q = 1}^{c\Delta^{-s/2}} (2q|I| + 2) \cdot \frac{\Delta^{s/2}}{q} \lesssim c|I| + \Delta^{s/2}\log(1/\Delta) \leq 2c|I|,   \end{displaymath}
recalling our assumption that $|I| \geq C\Delta^{s/2}\log(1/\Delta)$. Now, we choose $c > 0$ small enough (depending on the implicit constants above) that $|I_{\mathrm{bad}}| \leq \tfrac{1}{4}|I|$. Then $|\tfrac{1}{2}I \, \setminus \, I_{\mathrm{bad}}| \geq \tfrac{1}{4}|I|$, and 
\begin{displaymath} \tfrac{1}{2}I \, \setminus \, I_{\mathrm{bad}} \subset \bigcup_{q = c\Delta^{-s/2}}^{\Delta^{-s/2}} \bigcup_{p \in I_{q} \cap \Z} B\left(\tfrac{p}{q},\tfrac{\Delta^{s/2}}{q}\right) \subset \bigcup_{q = 1}^{\Delta^{-s/2}} \bigcup_{p \leq q} B\left(\tfrac{p}{q},c^{-1}\Delta^{s}\right). \end{displaymath}
Since each ball (or interval) on the right has length $2c^{-1}\Delta^{s}$, in order to cover $\tfrac{1}{2}I \, \setminus \, I_{\mathrm{bad}}$, we need at least $\sim |I|\Delta^{-s}$ \textbf{distinct} centres $p/q$. We may assume that $B(p/q,c^{-1}\Delta^{s}) \cap \tfrac{1}{2}I \neq \emptyset$ for all these $p/q$, which implies that $p/q \in I$ by our assumption $|I| \gg \Delta^{s/2}$. In summary, we have now found $\sim |I|\Delta^{-s}$ distinct rational numbers $p/q \in I$ with $p \leq q \leq \Delta^{-s/2}$. Since the separation of such rationals $p_{1}/q_{2} \neq p_{2}/q_{2}$ is automatically
\begin{displaymath} \left|\frac{p_{1}}{q_{1}} - \frac{p_{2}}{q_{2}} \right| \geq \frac{1}{q_{1}q_{2}} \geq \Delta^{s}, \end{displaymath}
we have proved the claim. \end{proof}

Finally, let us verify property (E), namely that if the sequence $\{\Delta_{n}\}$ decays so rapidly that $\Delta_{n} < \Delta_{n - 1}^{n}$ for all $n \geq 1$, then
\begin{equation}\label{form40} E = \bigcap_{n \in \N} (\cup \mathcal{E}_{\Delta_{n}}) = s. \end{equation}
There is a slight inconsistency in our notation: in \eqref{defEApp} we defined $\mathcal{E}_{\Delta_{n}}$ to consist of all shifted rationals of the form $p/q - \theta$, where $\theta = \Delta^{1 - \tau/2} = o_{\Delta \to 0}(1)$. However, in \eqref{form40} the set $\mathcal{E}_{\Delta_{n}}$ stands for the $\Delta_{n}$-neighbourhood of this set. In fact, in the argument below we will view and handle $\mathcal{E}_{\Delta_{n}}$ as a union of closed $\Delta_{n}$-intervals. Notice that Claim \ref{c1} and property (P2) are blind to such distinctions. 

The idea is to apply Lemma \ref{lemma6} to a suitably constructed Cantor subset of $E$. Namely, for sufficiently large indices $n \geq 1$ it follows from the rapid decay of $\{\Delta_{n}\}$ that
\begin{displaymath} \Delta_{n - 1} \geq C\Delta_{n}^{s/2}\log(1/\Delta_{n}). \end{displaymath}
For these indices Claim \ref{dirichletClaim} implies that
\begin{displaymath} |\mathcal{E}_{\Delta_{n}} \cap I|_{\Delta_{n}^{s}} \sim \Delta_{n - 1}\Delta_{n}^{-s}, \qquad I \in \mathcal{E}_{\Delta_{n - 1}} \text{ with } I \subset [0,1]. \end{displaymath}
In particular, applying once more the rapid decay of $\{\Delta_{n}\}$, we have $|\mathcal{E}_{\Delta_{n}} \cap I|_{\Delta_{n}^{s}} \geq (\Delta_{n - 1}/\Delta_{n})^{s - 1/n}$ for all these $I \in \mathcal{E}_{\Delta_{n - 1}}$, and for $n \geq 1$ sufficiently large. In particular $\mathcal{E}_{\Delta_{n}} \cap I$ is a relative $(\Delta_{n},s - \tfrac{1}{n},C)$-subset of $I$ for all $I \in \mathcal{E}_{\Delta_{n - 1}}$ with $I \subset [0,1]$, and $n \geq 1$ sufficiently large. Now \eqref{form40} follows from Lemma \ref{lemma6}.

\bibliographystyle{plain}
\bibliography{references}

\end{document}